\documentclass[12pt]{amsart}
\usepackage{amsaddr, amssymb }
\usepackage{enumitem}
\newcommand{\Cdb}{\mbox{$\mathbb{C}$}}

\newcommand{\Ndb}{\mbox{$\mathbb{N}$}}

\newcommand{\Rdb}{\mbox{$\mathbb{R}$}}

\newcommand{\A}{\mbox{${\mathcal A}$}}

\newcommand{\E}{\mbox{${\mathcal E}$}}
\newcommand{\F}{\mbox{${\mathcal F}$}}

\renewcommand{\H}{\mbox{${\mathcal H}$}}

\newcommand{\M}{\mbox{${\mathcal M}$}}
\newcommand{\N}{\mbox{${\mathcal N}$}}

\newcommand{\R}{\mbox{${\mathcal R}$}}
\renewcommand{\S}{\mbox{${\mathcal S}$}}

\newcommand{\tr}{\operatorname{tr}}
\newcommand{\Tr}{\operatorname{Tr}}
\newcommand{\Id}{\operatorname{Id}}
\newcommand{\La}{\Lambda}
\newcommand{\Ran}{\operatorname{Ran}}
\newcommand{\simp}{\underset{\geq 0}{\sim}}

\newtheorem{theorem}{Theorem}[section]
\newtheorem{lemma}[theorem]{Lemma}
\newtheorem{corollary}[theorem]{Corollary}
\newtheorem{proposition}[theorem]{Proposition}
\newtheorem{definition}[theorem]{Definition}
\theoremstyle{remark}
\newtheorem{remark}[theorem]{\bf Remark}
\theoremstyle{definition}

\numberwithin{equation}{section}

\begin{document}

\title[]{On the structure of contractively decomposable projections on noncommutative $L^p$-spaces and Schatten spaces}

\author{Estelle Boffy}
\address{Laboratoire de Math\'ematiques, Universit\'e Marie et Louis Pasteur
, 25030 Besan\c con Cedex, France}
\email{estelle.boffy@math.cnrs.fr}

\date{\today}

\begin{abstract} We show that the range of a contractively decomposable projection on a noncommutative Haagerup $L^p$-space, $L^p(\mathcal{M},\varphi)$,
 for $1<p<\infty$, is completely isometrically isomorphic to a corner of a noncommutative $L^p$-space, that is $eL^p(\mathcal{N},\psi)(1-e)$, with $e\in\mathcal{N}$ a projection. In the setting of Schatten spaces, we obtain a more precise description: the range of a contractively decomposable projection on $S^p(K,H)$ is isometric to an $\ell^p$ direct sum of subspaces of the form $S^p(K',H')$. Furthermore, we show that contractively 1-pseudo decomposable projections on Schatten spaces are automatically contractively decomposable, establishing the equivalence between these two notions in this setting.
\end{abstract}

\maketitle
{\it Keywords: Projections, complemented subspaces, Haagerup $L^p$-spaces, Schatten spaces, decomposable maps.

 }


\section{Introduction}\label{Intro}

The study of contractive projections and 1-complemented subspaces of Banach spaces is a major theme in the geometry of Banach spaces. Recall that a subspace $Y$ of a space $X$ is 1-complemented if it is the range of a linear contractive projection $P:X\to X$. For more information on this topic, we refer to the survey \cite{Randri}. In classical (i.e. commutative) $L^p$-spaces, a description of 1-complemented subspaces of $L^p(\Omega,\mathcal{F},\mu)$, for $(\Omega,\mathcal{F},\mu)$ a probability space, is given by Douglas for $p=1$, and by Ando for $1<p\neq 2<\infty$, in 1965 and 1966 respectively, see \cite{Ando}, and \cite{Douglas}. These results were generalized by Bernau and Lacey for an $L^p$-space on a general measured space (see \cite{Bernau_Lacey}). Note that in a smooth Banach space $X$, if $Y$ is a 1-complemented subspace of $X$, then there exists a \textit{unique} contractive projection $P:X\to X$ onto $Y$, see \cite[Theorem 6]{Cohen_Sullivan}. Hence, in $L^p$-spaces, for $1<p<\infty$, the study of contractive projections is equivalent to the study of 1-complemented subspaces.

In 1978 and in 1992, Arazy and Friedman described all 1-complemented subspaces of $S^p(K,H)$ in \cite{AF1} and \cite{AF2}, for $1\leq p\neq 2<\infty$ and for $H,K$ some separable Hilbert spaces. Recall that for $H$ and $K$ some Hilbert spaces, and for $1\leq p<\infty$, $S^p(K,H)$ denotes the Schatten space, that is the space of all bounded operators $x:K\to H$ such that $\Tr(|x|^p)<\infty$, where $\Tr$ denotes the usual trace on $B(K)$. In 2009, Le Merdy, Ricard and Roydor used the result of Arazy and Friedman to obtain a description of all  completely 1-complemented subspaces of $S^p(K,H)$, for $1\leq p\neq 2<\infty$, see \cite{MRR}. In 2026, a description of positively 1-complemented subspaces of $S^p(H):=S^p(H,H)$, for $1\leq p<\infty$, is given in \cite{B}, relying on the result of Arazy and Friedman. On a general noncommutative $L^p$-space, denoted by $L^p(\M)$ (resp. $L^p(\M,\varphi)$ for the Haagerup $L^p$-spaces), that is an $L^p$-space defined from a von Neumann algebra $\M$, equipped with a normal semifinite faithful trace $\tau$ (resp. weight $\varphi$), there is no description of 1-complemented subspaces. Nevertheless, significant progress has been made under an additional assumption of positivity on the contractive projection. In 2024, in \cite{AR}, Arhancet and Raynaud showed that the range of a 2-positive contractive projection on noncommutative $L^p$-space is completely order isometric to a noncommutative $L^p$-space. 
The contractive and positive projections in noncommutative $L^p$-spaces are studied by Arhancet in \cite{A_positiveproj}. It is proved that in a wide range of cases, the range of a contractive and positive projection on $L^p(M)$ is isometric to a nonassociative $L^p$-space associated with a $JW^*$-algebra.
Moreover, Arhancet studied the contractively decomposable projections on noncommutative $L^p$-spaces in \cite{A_dec}. In the latter paper, the notion of decomposable and $n$-pseudo decomposable map between noncommutative Haagerup $L^p$-spaces is introduced, we recall these definitions below.

The definition of decomposable map between noncommutative $L^p$-spaces is given in \cite{JungeRuan}. It is an adaptation of the classical notion of the case $p=\infty$. See \cite{Haa_dec} for more information of this classical notions, and \cite{Arh_Krieg} for the case $1\leq p<\infty$.

\begin{definition}\label{dec_on_Lp}
Let $1\leq p \leq \infty$, and let $T : L^p(\M,\varphi)\to L^p(\N,\psi)$ a linear map between noncommutative Haagerup $L^p$-spaces on von Neumann algebras $\M$ and $\N$ equipped with normal faithful state $\phi$ and $\psi$ respectively. We denote by $T^\circ$ the map defined by $T^\circ(x) := T(x^*)^*$, for any $x\in L^p(\M,\varphi)$. $T$ is called decomposable if there exist bounded linear maps $v_1,v_2 : L^p(\M,\varphi)\to L^p(\N,\psi)$ such that the map \begin{equation}\label{def_dec}
    \Phi := \left[\begin{array}{cc}
    v_1 & T \\
    T^\circ & v_2
\end{array}\right] : S^p_2(L^p(\M,\varphi))\to S^p_2(L^p(\N,\psi))
\end{equation} defined by \[\Phi\left(\left[\begin{array}{cc}
    a & b \\
    c & d
\end{array} \right]\right) = \left(\left[\begin{array}{cc}
    v_1(a) & T(b) \\
    T^\circ(c) & v_2(d)
\end{array} \right]\right)\] is completely positive. In this case, define \[\|T\|_\mathrm{dec} := \inf\max\{\|v_1\|,\|v_2\|\},\] where the infimum is taken over all maps $v_1$ and $v_2$ such that the map $\Phi$ given in (\ref{def_dec}) is completely positive.
\end{definition}
 The map $\|\cdot\|_\mathrm{dec}$ is a complete norm on the space of all decomposable maps from $L^p(\M,\varphi)$ to $L^p(\N,\psi)$. According to \cite[Proposition 3.5]{Arh_Krieg}, we can choose $v_1$ and $v_2$ such that $\|T\|_\mathrm{dec} = \max\{\|v_1\|,\|v_2\|\}$, that is the infimum in the definition of $\|\cdot\|_\mathrm{dec}$ is attained.
$T$ is contractively decomposable if \(\|T\|_\mathrm{dec}\leq 1.\)

\begin{definition}\label{def_n_pseudo_dec}
In the previous situation, if there exists a map $\Phi = \left[\begin{array}{cc}
    v_1 & T \\
    T^\circ & v_2
\end{array}\right]$ which is $n$-positive, for $1\leq n\leq\infty$, we say that $T$ is \textit{$n$-pseudo-decomposable}, and we can define \[\|T\|_{n,\mathrm{dec}} := \inf\max\{\|v_1\|,\|v_2\|\},\] where the infimum is taken over all such maps $v_1$ and $v_2$.
\end{definition} Similarly to the case $n=\infty$, this infimum is a minimum, we can choose $v_1$ and $v_2$ such that $\|T\|_{n,\mathrm{dec}}=\max\{\|v_1\|,\|v_2\|\}$, see \cite[Proposition 4.3]{A_dec}. This is a norm on the space of all $n$-pseudo-decomposable maps from $L^p(\M,\varphi)$ to $L^p(\N,\psi)$. This notion is introduced in \cite[Definition 4.1]{A_dec}. Note that if $n=\infty$, we recover Definition \ref{dec_on_Lp}. If $T : L^p(\M,\varphi)\to L^p(\N,\psi)$ is decomposable, then it is $n$-pseudo decomposable, for every $1\leq n\leq \infty$, and we have \(\|T\|_{n,\mathrm{dec}}\leq \|T\|_\mathrm{dec}.\)

The main result of \cite{A_dec} is that for a contractively $n$-pseudo decomposable projection on a noncommutative Haagerup $L^p$-spaces, $P:L^p(\M,\varphi)\to L^p(\M,\varphi)$, with $1\leq n\leq \infty$, $1<p<\infty$, $\M$ a $\sigma$-finite von Neumann algebra, and $\varphi$ a normal faithful state on $\M$, there exist positive elements $h,k\in L^p(\M,\varphi)$ with support projections denoted by $\operatorname{s}(h)$ and $\operatorname{s}(k)$ respectively, and there exists a weak$^*$ continuous contractively $n$-pseudo decomposable projection $w:\operatorname{s}(h)\M\operatorname{s}(k)\to \operatorname{s}(h)\M\operatorname{s}(k)$ such that \[P\left(h^\frac{1}{2}xk^\frac{1}{2}\right) = h^\frac{1}{2}w(x)k^\frac{1}{2},\quad x\in \operatorname{s}(h)\M\operatorname{s}(k).\]

In the present paper, we rely on results contained in \cite{A_dec} to show that the range of a contractively decomposable projection can be seen as a corner of a noncommutative $L^p$-space. We only consider $\sigma$-finite von Neumann algebras. Note that any $\sigma$-finite von Neumann algebra admits a faithful normal state. Here is the first main result of this paper:

\begin{theorem}\label{main1}
Let $1<p<\infty$, let $\M$ be a $\sigma$-finite von Neumann algebra equipped with a normal faithful state $\varphi$. Consider a bounded map \(P: L^p(\M,\varphi)\to L^p(\M,\varphi)\).
If $P$ is a contractively decomposable projection, then there exists $\N$ a $\sigma$-finite von Neumann algebra, equipped with a normal faithful positive linear form $\tilde{\phi}$, and a projection $e_1\in \N$ such that the subspace $\Ran(P)$ is completely isometrically isomorphic to \(e_1L^p(\N,\tilde\phi)(1-e_1).\)

\end{theorem}

Note that this theorem implies that the range of a contractively decomposable projection is an interpolation space associated with the $W^*$-TRO $e_1\N(1-e_1)$. See \cite[Remark 5.6]{A_dec}. Moreover, a subspace of the form $eL^p(\N,\psi)(1-e)$, with $e\in\M$ an orthogonal projection, is the range of a contractively decomposable projection. This result is detailed in \cite[Proposition 5.2]{A_dec}.

To prove Theorem \ref{main1}, we follow the proof of the main Theorem of \cite{A_dec} (see \cite[Theorem 1.1]{A_dec}) to obtain a completely positive projection \[\Phi := \left[\begin{array}{cc}
    P_1 & P \\
    P^\circ & P_2
\end{array}\right]: S^p_2\left(L^p(\M,\varphi)\right)\to S^p_2\left(L^p(\M,\varphi)\right),\] with $P_1,P_2 : L^p(\M, \varphi)\to L^p(\M,\varphi)$ some completely positive contractive projections. The main novelty is showing that $\Phi$ is contractive. Once this is established, we apply \cite[Theorem 1.1]{AR} to obtain the description of $\Ran(\Phi)$, and deduce the description of $\Ran(P)$.

In the setting of Schatten spaces, we obtain a more precise description of contractively decomposable projections, relying on \cite{B}. Note that the definitions of $n$-pseudo decomposable map may be adapted for an operator $T:S^p(K,H)\to S^p(K',H')$, see Definition \ref{dec_on_Sp} in Section \ref{section dec Sp}. We show, in particular, that if $P:S^p(K,H)\to S^p(K,H)$ is a contractively decomposable projection, with $1<p<\infty$, $K,H$ some separable Hilbert spaces, then the range of $P$ is isometric to an $\ell^p$ direct sum of subspaces of the form $S^p(K',H')$. Note that for $p\neq 2$, the result directly follows from the paper \cite{MRR}, since a contractively decomposable map is completely contractive, but it is new for $p=2$. The major part of the result on Schatten spaces is to show that if $P:S^p(K,H)\to S^p(K,H)$ is a contractively 1-pseudo decomposable projection, with $1<p<\infty$, then it is a contractively decomposable projection. The result is stated in the following theorem:

\begin{theorem}\label{main2}
Let $H$, $K$ be Hilbert spaces, let $1<p <\infty$, and let $X$ be a subspace of $S^p(K,H)$. The following assertions are equivalent.
\begin{enumerate}
    \item $X$ is 1-complemented with a contractively 1-pseudo decomposable projection,
    \item $X$ is 1-complemented with a contractively decomposable projection,
    \item there exist, for a set $A$, two families of indices $(I_\alpha)_{\alpha\in A}$ and $(J_\alpha)_{\alpha \in A}$, a family of Hilbert spaces $(H_\alpha)_{\alpha\in A}$, some positive and injective operators $a_\alpha \in S^p(K_\alpha)$ and two isometries \[U : \underset{\alpha\in A}{\overset{2}{\oplus}} \ell^2_{J_\alpha}(K_\alpha)\to K\quad \text{ and }\quad V : \underset{\alpha\in A}{\overset{2}{\oplus}} \ell^2_{I_\alpha}(K_\alpha)\to H\] such that \[X = V\left(\bigoplus_{\alpha\in A}^p S^p_{I_\alpha,J_\alpha}\otimes a_\alpha\right)U^*.\]
\end{enumerate}
\end{theorem}


\vspace{0.5cm}

\textbf{Structure of the paper} 
\begin{itemize}
    \item \textbf{Section 2} reviews the definition and properties of noncommutative Haagerup $L^p$-spaces.
    
    \item \textbf{Section 3} is devoted to the description of contractively decomposable projection on noncommutative Haagerup $L^p$-spaces. We recall the lifting theorem from \cite{AR}, some results about the support projection of the range of a contractive and positive projection, and properties about contractively $n$-pseudo decomposable projections from \cite{A_dec}. Then, we prove Theorem \ref{main1}.
    
    \item \textbf{Section 4} investigates the contractively $n$-pseudo decomposable projections in Schatten spaces. The first subsection gives some preliminaries results about decomposable (and $n$-pseudo decomposable) maps on $S^p(K,H)$, and then recalls the description of positively 1-complemented subspaces of $S^p(H)$, from \cite{B}. A second subsection gives the description of the subspaces of $S^p(H)$ that are the range of a completely positive and contractive projection. This result is used in the proof of the part $(3)\Rightarrow (2)$ of Theorem \ref{main2}. The last subsection studies the $M_2$-splitting positive contractive projections on Schatten spaces, to finally obtain the description of the range of a contractively 1-pseudo decomposable projection in $S^p(K,H)$ (which is the part $(1)\Rightarrow (3)$ of Theorem \ref{main2}).
    
\end{itemize}







\section{Preliminaries on noncommtative Haagerup $L_p$ spaces}


In this subsection, we briefly recall the construction of the noncommutative Haagerup $L^p$-spaces associated with a von Neumann algebra and a normal faithful positive linear functional $\varphi$. The construction of such spaces can actually be done if $\varphi$ is a normal faithful weight on $\M$, the description is a little simpler in the case of a linear functional. We refer the reader to \cite{HJX}, \cite{Hiai} and \cite{Terp} for details.

Let $\M\subset B(H)$ be a von Neumann algebra on a Hilbert space $H$, let $\varphi$ be a normal faithful positive linear functional on $\M$. We denote by $\sigma^\varphi = (\sigma^\varphi_t)_{t\in\mathbb R}$ the one-parameter modular automorphism group associated to $\varphi$ (see \cite[Chap. VIII]{Takesaki2}). Let $\R := \M \rtimes_{\sigma^\varphi}\Rdb\subset B(L^2(\Rdb,H))$, be the crossed product of $\M$ by $\Rdb$ via $\sigma^\varphi$ (see \cite[Chap. X]{Takesaki2}). We may identify $\M$ with a subalgebra of $\R$.

We define the dual action $\widehat{\sigma^\varphi} : \Rdb\to Aut(\R)$ of $\sigma^\varphi$ by \[\widehat{\sigma^\varphi_t}(x) := W(t)xW(t)^*,\quad t\in\Rdb,~x\in \R,\] where $W(t)\in B(L^2(\Rdb,H))$ is defined by $\left[W(t)\xi\right](s) := e^{-its}\xi(s)$, for $s,t\in\Rdb$ and $\xi\in L^2(\Rdb,H)$.
Note that \begin{equation}\label{M=Linfini}
    \M = \left\{x\in\R,\quad \widehat{\sigma^\varphi_t}(x) = x,\quad \text{for every }t\in\Rdb\right\}.
\end{equation}

There exists a unique normal semifinite faithful trace $\tau_\varphi = \tau$ on $\R$ such that \(\tau\circ \widehat{\sigma^\varphi_t} = e^{-t}\tau\) for every $t\in\Rdb$ (see \cite[Theorem 8.15]{Hiai}). We consider the $*$-algebra $L^0(\R,\tau)$ of $\tau$-measurable operators (see \cite[Chap. 4]{Hiai}), and for $1\leq p\leq \infty$, define \[L^p(\M,\varphi) := \left\{y\in L^0(\R,\tau),\quad \widehat{\sigma^\varphi_t}(y) = e^{-t/p}y,~\text{ for every }t\in\Rdb\right\}.\] This is a $*$-subalgebra of $L^0(\R,\tau)$. Note that according to (\ref{M=Linfini}), $L^\infty(\M,\varphi)=\M$. We can define a positive cone for this subspace: \[L^p(\M,\varphi)_+:=L^p(\M,\varphi)\cap L^0(\R,\tau)_+.\] Consider $x\in L^p(\M,\varphi)$ and let $x = u|x|$ its polar decomposition. Then, $u\in \M,$ and $|x|\in L^p(\M,\varphi)$. We denote by $s_\ell(x)\in\M$ and $s_r(x)\in\M$ the right and left support projection of an element $x\in L^p(\M,\varphi)$. 

Consider $\psi\in\M^+_*$ a normal positive linear functional on $\M$. It induces a normal semifinite faithful weight $\widehat{\psi}$ on $\R$ called the dual weight (see \cite[Chap. X,§1]{Takesaki2}). We denote $h_\psi$ the Radon-Nikodym derivative of $\widehat{\psi}$ with respect to $\tau$. Then $h_\psi \in L^1(\M,\varphi)_+$ and \[\widehat{\psi}(x) = \tau\left(h_\psi^{\frac{1}{2}} xh_\psi^{\frac{1}{2}}\right),\quad x\in\R_+.\] Moreover, the map $\psi\mapsto h_\psi$ is a bijection from $\M^+_*$ onto $L^1(\M,\varphi)_+$ which naturally extends to an isomorphism from $\M_*$ onto $L^1(\M,\varphi)$ (also denoted by $\psi\mapsto h_\psi$).

We equip the space $L^1(\M,\varphi)$ with the norm inherited from $\M_*$: \[\|h_\psi\|_1:=\|\psi\|_{\mathcal{M}_*},\quad \psi \in \M_*.\] 
Then, we define a norm on $L^p(\M,\varphi)$, for $x\in L^p(\M,\varphi)$, set \[\|x\|_p := \left\| |x|^p\right\|_1^{1/p}.\]
Note that if $\varphi$ is tracial, that is \(\varphi(x^*x)=\varphi(xx^*),\) for $x\in\M$, then the space $L^p(\M,\varphi)$ defined here coïncide isometrically with the classical tracial noncommutative $L^p$-space (see \cite[p.62]{Terp}. 

Using the isomorphism $\psi\in\M_*\mapsto h_\psi\in L^1(\M,\varphi)$, we can define a linear functional \(\Tr_\varphi : L^1(\M,\varphi)\to \Cdb\) defined by \begin{equation}
    \Tr_\varphi(h_\psi) := \psi(1).
\end{equation}
This linear functional is continuous on $L^1(\M,\varphi)$, for every $x\in L^1(\M,\varphi)$, \[\left|\Tr_\varphi(x)\right|\leq \|x\|_1.\] 
Moreover, we have a version of Hölder's inequality. Let $1\leq p,q,r\leq \infty$ with $\frac{1}{p}+\frac{1}{q}=\frac{1}{r}$, let $x\in L^p(\M,\varphi)$ and let $y\in L^q(\M,\varphi)$, then $xy\in L^r(\M,\varphi)$ and \[\|xy\|_r\leq \|x\|_p\|y\|_q.\] In particular, $L^p(\M,\varphi)$ is an $\M$-bimodule for all $1\leq p\leq \infty$. Moreover, if $\frac{1}{p}+\frac{1}{q}=1$, \[\left|\Tr_\varphi(xy)\right|\leq\|xy\|_1\leq \|x\|_p\|y\|_q,\quad \text{ for }x\in L^p(\M,\varphi),~y\in L^q(\M,\varphi),\]
and the linear map $\Tr_\varphi$ has, in addition, a tracial property: \[\Tr_\varphi(xy)=\Tr_\varphi(yx),\quad \text{for every }x\in L^p(\M,\varphi),~y\in L^q(\M,\varphi).\] Moreover, if $1\leq p<\infty$ and $\frac{1}{p}+\frac{1}{q}=1$, the space $L^q(\M,\varphi)$ is isometrically the dual of $L^p(\M,\varphi)$ under the duality bracket defined by \[(x,y)\in L^p(\M,\varphi)\times L^q(\M,\varphi)\mapsto \Tr_\varphi(xy)\in\Cdb.\]
Note also that in case $p=2$, the space $L^2(\M,\varphi)$ is a Hilbert space with the inner product \( \left(x~|~y\right) := \Tr_\varphi(y^*x).\)

 We denote $D_\varphi$ the Radon-Nikodym derivative of $\widehat{\varphi}$ with respect to $\tau$, that is \(D_\varphi:= h_\varphi \in L^1(\M,\varphi)_+\). It is called the density operator of $\varphi$.  For every $x\in \M$, we have \begin{equation}\label{D_op_densite}
     \varphi(x) = \Tr_\varphi(D_\varphi x) = \Tr_\varphi(xD_\varphi).
 \end{equation}



For the following three results, we refer to \cite{AR} and \cite{HJX}.


\begin{proposition}\label{change_of_weight}
    Let $\M$ be a von Neuman algebra, let $\varphi_1$ and $\varphi_2$ be two normal faithful positive linear forms on $\M$.
    Then there exists a $*$-isomorphism $\kappa : \M\rtimes_{\sigma^{\varphi_1}}\Rdb\to \M\rtimes_{\sigma^{\varphi_2}}\Rdb$ such that 
    \begin{itemize}
        \item the map $\kappa$ preserves $\M$.
        \item The map $\kappa$ extends to a topological $*$-isomorphism between the algebras of measurable operators $\tilde{\kappa} : L^0(\M\rtimes_{\sigma^{\varphi_1}}\Rdb,\tau_1)\to L^0(\M\rtimes_{\sigma^{\varphi_2}}\Rdb,\tau_2)$.
        \item The map $\tilde{\kappa}$ restricts into a isometric $*$-isomorphism $\kappa_1 : L^1(\M,\varphi_1)\to L^1(\M,\varphi_2)$ that preserves the $\M$-bimodular structure and preserves the trace, that is \[\Tr_{\varphi_1} = \Tr_{\varphi_2}\circ \kappa_1.\]
        \item The map $\tilde{\kappa}$ restricts into a completely order isometric isomorphism $\kappa_p : L^p(\M,\varphi_1)\to L^p(\M,\varphi_2)$.
    \end{itemize}
\end{proposition}
We recall the notion of completely isometric maps and completely positive maps between noncommutative $L^p$-spaces in the last paragraph of this section.

\begin{theorem}\label{HJX}
Let $\M$ be a von Neumann algebra, equipped with a normal faithful linear form $\varphi$. Let $T : \M\to \M$ be a unital positive map such that $$\varphi(T(x)) \leq \varphi(x),\quad x\in\M_+.$$ Denote by $D_\varphi \in L_1(\M,\varphi)_+$ the density operator associated to $\varphi$. 
Then for $1\leq p<\infty$, the map $$T_p : \begin{array}{ccc}
    D_\varphi^{1/2p}\M D_\varphi^{1/2p} & \to & D_\varphi^{1/2p}\M D_\varphi^{1/2p} \\
    D_\varphi^{1/2p} x D_\varphi^{1/2p} & \mapsto & D_\varphi^{1/2p}T(x) D_\varphi^{1/2p}\end{array}$$ uniquely extends to a positive and contractive map from $L^p(\M,\varphi)$ into $L^p(\M,\varphi)$.
\end{theorem}

For $\M$ a von Neumann algebra equipped with a normal faithful positive linear form $\varphi$, we denote by $\M^\varphi := \left\{x\in\M,\quad \sigma_t^\varphi(x)=x,~\text{ for every }t\in\Rdb\right\}$ the centralizer of $\varphi$. Recall that if $x\in \M$, then \[x\in\M^\varphi \iff \varphi(xy)=\varphi(yx),\quad\text{ for any }y\in\M.\]

For $e\in\M^\varphi$ a projection, by applying Theorem \ref{HJX} to the canonical inclusion $e\M e\mapsto\M$, we obtain the following proposition.
\begin{proposition}\label{eL^pe}
    Let $\M$ be a von Neumann algebra equipped with a normal faithful positive linear form $\varphi$.
    If $e$ is an orthogonal projection such that $e\in\M^\varphi$, we have an identification between the Banach spaces $L^p(e\M e, \varphi_e)$ and $eL^p(\M,\varphi)e$, where $\varphi_e$ is the restriction of $\varphi$ on $e\M e$.
   
    Moreover, the Haagerup trace $\Tr_\varphi$ restricts to $\Tr_{\varphi_e}$ on $L^1(e\M e)$.
\end{proposition}

In the remainder of the paper, we work with contractive and $n$-positive operators between noncommutative $L^p$-spaces, for $1\leq n\leq \infty$. To understand this notion, we rely on the identification \[S^p_n\left(L^p(\M,\varphi)\right) = L^p\left(M_n(\M),\tr_n\otimes\varphi\right),\] where $\tr_n$ is the usual trace on $M_n(\Cdb)$. We refer to \cite{JungeRuanXu} for details. Note that $S^p_n := S^p(\ell^2_n)$.

Consider a map $T:L^p(\M,\varphi)\to L^p(\N,\psi)$, with $\M,\N$ some von Neumann algebras and $\varphi,\psi$ some normal faithful positive linear functionals on $\M$ and $\N$ respectively. For $1\leq n < \infty$, we say that $T$ is $n$-positive when the map $\Id_{S^p_n}\otimes T : S^p_n(L^p(\M,\varphi))\to S^p_n(L^p(\N,\psi))$ is positive. If $\Id_{S^p_n}\otimes T$ is positive for all $n$, we say that $T$ is completely positive. We say that $T$ is $n$-bounded when the map $\Id_{S^p_n}\otimes T : S^p_n(L^p(\M,\varphi))\to S^p_n(L^p(\N,\psi))$ is bounded. If for every $n$, the maps $\Id_{S^p_n}\otimes T $ are uniformly bounded in $n$, we say that $T$ is completely bounded. If for every $n$, $\Id_{S^p_n}\otimes T $ is an isometry, we say that $T$ is completely isometric. The map $T$ is a completely order isometric isomorphism if it is bijective and $T$ and $T^{-1}$ are completely positive and completely isometric.

\section{Contractive projections in noncommutative $L_p$ spaces}
\subsection{Preliminaries on contractive projections}

The following theorem, from \cite[Theorem 5.1]{AR} is inspired by \cite[Theorem 3.1]{JungeRuanXu}.
\begin{theorem}\label{lifting thm}
Let $\M$ be a von Neumann algebra, equipped a normal faithful state $\varphi$. Let $1\leq p<\infty$, let $T : L^p(\M,\varphi)\to L^p(\M,\varphi)$ be a positive map. Let $h$ be a positive element in $L^p(\M,\varphi)$. Then, there exists a unique map $v:\M\to \operatorname{s}(T(h))\M \operatorname{s}(T(h))$ such that \[T(h^{1/2}xh^{1/2}) = T(h)^{1/2}v(x)T(h)^{1/2},\quad x\in\M,\] where $\operatorname{s}(T(h))$ is the support projection of $T(h)$. Moreover, $v$ is a unital contractive and normal map. If $T$ is $n$-positive, for $1\leq n \leq \infty$, then $v$ is also $n$-positive.
\end{theorem}

Combined with the following proposition, we may obtain a more precised lifting theorem in the case of positive and contractive projections. Let $P : L^p(\M,\varphi)\to L^p(\M,\varphi)$ be a bounded projection, 
and $1\leq p<\infty$. Define its support projections \[\operatorname{s}_\ell(P):= \sup\{\operatorname{s}_\ell(x),\quad x\in \Ran(P)\},\quad \operatorname{s}_r(P):= \sup\{\operatorname{s}_r(x),\quad x\in \Ran(P)\}.\] The subspace $\Ran(P)$ is called non-degenerate if $\operatorname{s}_\ell(P)$ and $\operatorname{s}_r(P)$ are the identity.
If $P$ is adjoint preserving, i.e. $P(x^*)=P(x)^*$ for every $x$, then $$\operatorname{s}_\ell(P) = \operatorname{s}_r(P),$$ and we denote this projection by $\operatorname{s}(P)$. Note that a positive map $T:L^p(\M,\varphi)\to L^p(\M,\varphi)$ is adjoint preserving.
The following proposition, from \cite[Proposition 6.1]{AR}, highlights a particular element in $\Ran(P)$ with a maximal support.
\begin{proposition}\label{existence_h}
Let $P : L^p(\M,\varphi)\to L^p(\M,\varphi)$ be a contractive positive projection, with $\M$ a $\sigma$-finite von Neumann algebra equipped with a normal faithful state $\varphi$, and $1\leq p<\infty$. Then there exits a positive element $h$ of $\Ran(P)$ such that $\operatorname{s}(h) = \operatorname{s}(P)$.

\end{proposition}

If we consider a contractively $n$-pseudo decomposable projection $P$ on $L^p(\M,\varphi)$, for $1\leq n\leq \infty$, then the $n$-positive map $\Phi : S^p_2(L^p(\M,\varphi))\to S^p_2(L^p(\M,\varphi))$ that comes from (\ref{def_dec}) may be chosen to be a projection. This results is proved in \cite[Proposition 7.1, Lemma 7.4]{A_dec}, and reviewed here.

\begin{proposition}\label{Phi_proj} Let $\M$ be a von Neumann algebra, equipped with a normal faithful state $\varphi$. Let $1< p<\infty$, and let $P: L^p(\M,\varphi)\to L^p(\M,\varphi)$ be a contractively $n$-pseudo-decomposable projection, for $1\leq n\leq \infty$. Then, there exist $n$-positive contractive projections $P_1,P_2 : L^p(\M,\varphi)\to L^p(\M,\varphi)$ such that the linear map \[\Phi := \left[\begin{array}{cc}
    P_1 & P \\
    P^\circ & P_2
\end{array}\right] : S^p_2(L^p(\M,\varphi))\to S^p_2(L^p(\M,\varphi))\] is a $n$-positive projection.

Moreover, let $e = \left[\begin{array}{cc}
    \operatorname{s}(P_1) & 0 \\
    0 & \operatorname{s}(P_2)
\end{array}\right]$, then $\Phi(exe)=\Phi(x)$, for every $x\in S^p_2(L^p(\M,\varphi))$.
\end{proposition}


\begin{remark}\label{ran(P)nondegenimpliesran(Phi)nondegen}\begin{enumerate}
    \item With the notation of the previous proposition, since $\Phi(exe)=\Phi(x)$, for $x\in S^p_2(L^p(\M,\varphi))$, we deduce that $P(\operatorname{s}(P_1)z\operatorname{s}(P_2)) = P(z)$ for every $z\in L^p(\M,\varphi)$. It follows that $\operatorname{s}_\ell(P)\leq \operatorname{s}(P_1)$ and $\operatorname{s}_r(P)\leq \operatorname{s}(P_2)$. Indeed, let $z\in\Ran(P)$, we have \[\|z\|_p = \|P(z)\|_p = \|P(\operatorname{s}(P_1)z\operatorname{s}(P_2)\|_p \leq \|\operatorname{s}(P_1)z\operatorname{s}(P_2)\|_p \leq \|z\|_p.\] Hence $\|z\|_p = \|\operatorname{s}(P_1)z\operatorname{s}(P_2)\|_p$. Since $y\mapsto \operatorname{s}(P_1)y\operatorname{s}(P_2)$ is a contractive projection on $L^p(\M,\varphi)$, we can use \cite[(iii) of Proposition 1.1]{AF1}, which adapts in the space $L^p(\M,\varphi)$ since it only uses the strict convexity of the space. We deduce that for every $z\in\Ran(P)$, $z=\operatorname{s}(P_1)z\operatorname{s}(P_2)$, hence the result.
    
    \item We deduce that if $\Ran(P)$ is non-degenerate, then $\Ran(P_1)$ and $\Ran(P_2)$ are non-degenerate, hence $\Ran(\Phi)$ is non-degenerate. 
    
    Conversely, if $\Ran(\Phi)$ is non-degenerate, since $\Phi(exe)=\Phi(x)$, for every $x\in S^p_2(L^p(\M,\varphi))$, then $1 = \operatorname{s}(\Phi) \leq e$, hence $e=1$ and $\operatorname{s}(P_1)=\operatorname{s}(P_2)=1$, so $\Ran(P_1)$ and $\Ran(P_2)$ are non-degenerate. Note that this does not imply that $\Ran(P)$ is non-degenerate.
\end{enumerate}
\end{remark}

\subsection{Description of the range of a contractively decomposable projection on noncommutative $L_p$}

As explained in \cite[Remark 3.6]{Arh_Krieg}, a delicate problem is that if $T:L^p(\M,\varphi)\to L^p(\M,\varphi)$ is a contractively ($n$-pseudo) decomposable map, we do not know if we can find a completely positive (or $n$-positive) map $\Phi$ as in (\ref{def_dec}) which is contractive. The following proposition states that if $T$ is a projection, this is indeed the case.

\begin{proposition}\label{phi_contr_Lp}Let $\M$ be a von Neumann algebra, equipped with a normal faithful state $\varphi$. Let $1< p<\infty$, let $P: L^p(\M, \varphi)\to L^p(\M,\varphi)$ be a contractively $n$-pseudo-decomposable projection, let  $P_1,P_2 : L^p(\M,\varphi)\to L^p(\M,\varphi)$ be $n$-positive contractive projections such that the linear map \[\Phi := \left[\begin{array}{cc}
    P_1 & P \\
    P^\circ & P_2
\end{array}\right] : S^p_2(L^p(\M,\varphi))\to S^p_2(L^p(\M,\varphi))\] is a $n$-positive projection.
Then the map $\Phi$ is contractive.
\end{proposition}
\begin{proof} We rely on the proof of \cite[Theorem 7.3]{A_dec}.
Recall that \begin{equation}\label{poids_sur_S^p_2(Lp)}
    S^p_2\left(L^p(\M,\varphi)\right) = L^p\left(M_2(\M),\tr_2\otimes\varphi\right),
\end{equation} where $\tr_2$ denotes the usual trace on $M_2(\Cdb)$.
Consider a positive element $h\in \Ran(P_1)$, with $\operatorname{s}(h)=\operatorname{s}(P_1)$, (resp. $k\in\Ran(P_2)$, $k\geq 0$, $\operatorname{s}(k)=\operatorname{s}(P_2)$). These elements exist according to Proposition \ref{existence_h}. Define $H := \left[\begin{array}{cc}
    h &0  \\
    0 & k
\end{array}\right]$. Note that $H\in \Ran(\Phi)$, it is positive, and $e:=\operatorname{s}(H) = \left[\begin{array}{cc}
    \operatorname{s}(h) &0  \\
    0 & \operatorname{s}(k)
\end{array}\right]\in M_2(\M)$.
Let $\varphi_2 := \tr_2\otimes\varphi$ and define another normal faithful positive linear form on $M_2(\M)$: \[\phi(x) := \varphi_2\left(exe +(1-e)x(1-e)\right), \quad x\in M_2(\M).\]
Note that $e$ belongs to the centralizer of $\phi$. According to Proposition \ref{eL^pe}, we have the identification \[eL^p(M_2(\M),\phi)e = L^p(eM_2(\M)e,\phi_e),\] where $\phi_e(x)=\phi(exe)=\varphi_2(exe)$. Let $$\N:= eM_2(\M)e = \left[\begin{array}{cc}
    \operatorname{s}(h)\M \operatorname{s}(h) & \operatorname{s}(h)\M \operatorname{s}(k) \\
    \operatorname{s}(k)\M \operatorname{s}(h) & \operatorname{s}(k)\M \operatorname{s}(k)
\end{array}\right].$$ 
We apply the lifting theorem (see Theorem \ref{lifting thm}) to the $n$-positive map $\Phi : L^p(M_2(\M),\varphi_2)\to L^p(M_2(\M),\varphi_2)$ with respect to the positive element $H$. Then, we obtain a normal unital contractive and $n$-positive map $Q:M_2(\M)\to \N$ such that \begin{equation}\label{defQ}
    \Phi(H^{\frac{1}{2}}xH^{\frac{1}{2}}) = H^{\frac{1}{2}}Q(x)H^{\frac{1}{2}},\quad x\in M_2(\M).
\end{equation}
According to \cite[Lemma 7.6 and Lemma 7.7]{A_dec}, $Q_{|\mathcal{N}}:\N\to\N$ is a faithful projection.
We define a normal faithful positive linear form on $\N$, denoted by $\psi : \N\to\Cdb$ as \[\psi(x) = \Tr_{\varphi_2}(H^px),\] where $\Tr_{\varphi_2}$ is the Haagerup trace on $L^1(M_2(\M),\varphi_2)$. Note that $H\in L^p(M_2(\M),\varphi_2)$, hence $H^px\in L^1(M_2(\M),\varphi_2)$ for any $x\in\N\subset M_2(\M)$. Since $\Phi^*(H^{p-1}) = H^{p-1}$ (see \cite[Lemma 2.5 and Lemma 2.6]{A_dec}), we have $$\psi\circ Q = \psi\quad \text{on }\N.$$ See \cite[Lemma 7.8]{A_dec} for the details. Applying Theorem \ref{HJX} to the map $Q_{|\mathcal{N}}:\N\to\N$, we obtain that the map $Q_p : D_\psi^{1/2p}\N D_\psi^{1/2p}\to L^p(\N,\psi)$ defined by \begin{equation}\label{Q_p}
    Q_p(D_\psi^{1/2p}xD_\psi^{1/2p}) = D_\psi^{1/2p}Q(x)D_\psi^{1/2p},
\end{equation} extends to a $n$-positive contractive map $Q_p : L^p(\N,\psi)\to L^p(\N,\psi)$. The operator $D_\psi \in L^1(\N,\psi)$ denotes the density operator associated to $\psi$. 
Since $Q$ is a projection, we can easily check that $Q_p$ is also a projection. Now, the goal is to compare the maps $Q_p$ and $\Phi$.

By Proposition \ref{change_of_weight}, we have two completely order identifications: \[\kappa_p : L^p(\N,\psi)\to L^p(\N,\phi_e)\quad\text{and}\quad \sigma_p : L^p(M_2(\M),\varphi_2)\to L^p(M_2(\M),\phi),\] that are restrictions of two topological $*$-isomorphisms, $\tilde{\kappa}$ and $\tilde{\sigma}$, between the algebras of measurable operators . By convenience, we now denote by $\kappa$ and $\sigma$ the maps $\tilde{\kappa}$ and $\tilde{\sigma}$ respectively, as well as their restrictions on $L^p$ and $L^1$. 
Then, since we identify $L^1(\N,\phi_e)$ with the subspace $eL^1(M_2(\M),\phi)e$ of $L^1(M_2(\M),\phi)$ (see Proposition \ref{eL^pe}), we have $\sigma(H^p) \in L^1(\N,\phi_e)$ and
\begin{equation}\label{D=}
    D_\psi = \kappa^{-1}\left(\sigma(H)^p\right) = \kappa^{-1}\left(\sigma(H^p)\right).
\end{equation} Indeed, 
for any $x\in\N\subset M_2(\M)$, we have \begin{equation*}
    \begin{split}
    \psi(x) = \Tr_{\varphi_2}(H^px) \overset{\text{Prop. } \ref{change_of_weight}}{=} \left(\Tr_\phi \circ \sigma\right)(H^px) = \Tr_\phi\left(\sigma(H^px)\right).
\end{split}
\end{equation*}
We deduce \begin{equation}\label{psi(x)=}
    \psi(x) = \Tr_\phi\left(\sigma(H)^px\right).
\end{equation}
Note that $e=\operatorname{s}(H)$, hence $eH^pe = H^p$. Since the map $\sigma$ is the restriction of a topological $*$-isomorphism that preserves $M_2(\M)$ (see Proposition \ref{change_of_weight}), we have $$\sigma(H)^p= \sigma(H^p) = \sigma(eH^pe) = e\sigma(H^p)e = e\sigma(H)^pe \in eL^1(M_2(\M),\phi)e.$$ We deduce that for any $x\in\N$, $\sigma(H)^px = e\sigma(H)^pxe \in eL^1(M_2(\M)\phi)e = L^1(\N,\phi_e)$.
Hence we have, for $x\in\N$,
\begin{equation*}
    \begin{split}
\psi(x) \overset{(\ref{psi(x)=})}{=} \Tr_\phi\left(\sigma(H)^px)\right) &\overset{\text{Prop. }\ref{eL^pe}}{=}  \Tr_{\phi_e}\left(\sigma(H)^px)\right) \\ &\overset{\text{Prop. }\ref{change_of_weight}}{=} \left(\Tr_\psi\circ\kappa^{-1}\right)(\sigma(H)^px)\\ &= \Tr_\psi\left(\kappa^{-1}\left(\sigma(H)^p\right)x\right).
\end{split}
\end{equation*}
The last equality holds because $x\in\N$, hence $$\kappa^{-1}(\sigma(H)^px) = \kappa^{-1}\left(\sigma(H)^p\right)x.$$
Since $\psi(x) = \Tr_\psi(D_\psi x)$ for any $x\in\N$, we deduce (\ref{D=}). From this equality, we deduce that \begin{equation}\label{H=}
    H = \sigma^{-1}\left(\kappa(D_\psi)\right)^{1/p} = \sigma^{-1}\left(\kappa(D_\psi^{1/p})\right).
\end{equation}

For any $x\in M_2(\M)$, we have $H^{1/2}xH^{1/2}\in L^p(M_2(\M),\varphi_2)$, and \begin{equation*} \begin{split}
\Phi(H^{1/2}xH^{1/2}) = H^{1/2}Q(x)H^{1/2} &\overset{(\ref{H=})}{=} \sigma^{-1}\left(\kappa(D_\psi^{1/2p})\right)Q(x)\sigma^{-1}\left(\kappa(D_\psi^{1/2p})\right)\\ &= \sigma^{-1}\left(\kappa\left(D_\psi^{\frac{1}{2p}}Q(x)D_\psi^{\frac{1}{2p}}\right)\right)\\ &\overset{(\ref{Q_p})}{=} \sigma^{-1}\left(\kappa\left(Q_p(D_\psi^{\frac{1}{2p}}xD_\psi^{\frac{1}{2p}})\right)\right)\\ &\overset{(\ref{D=})}{=}\sigma^{-1}\left(\kappa\left(Q_p\left(\kappa^{-1}\left(\sigma(H)^{\frac{1}{2}}\right)x\kappa^{-1}\left(\sigma(H)^{\frac{1}{2}}\right)\right)\right)\right)\\
&= \sigma^{-1}\left(\kappa\left(Q_p\left(\kappa^{-1}\left(\sigma\left(H^{\frac{1}{2}}xH^{\frac{1}{2}}\right)\right)\right)\right)\right).
\end{split}\end{equation*}
Note that $Q(x)\in\mathcal{N}$, hence $(\sigma^{-1}\left(\kappa(Q(x))\right) = Q(x)$, this explains the third equality.
Since the subspace $H^{1/2}M_2(\M)H^{1/2}$ is dense in $eL^p(M_2(\M),\varphi_2)e$, we obtain the following equality:  \[\sigma^{-1}\left(\kappa\left(Q_p\left( \kappa^{-1}\left(\sigma(x)\right)\right)\right)\right) = \Phi(x),\quad x\in eL^p(M_2(\M),\varphi_2)e.\]
Now, set $\Psi : x\in L^p(M_2(\M),\varphi_2) \mapsto exe\in eL^p(M_2(\M),\varphi_2)e$. It is a contraction, and since $\Phi(x) = \Phi(exe)$ for any $x\in L^p(M_2(\M),\varphi_2)$, we have \begin{equation}\label{Phi=}
    \sigma^{-1}(\kappa(Q_p(\kappa^{-1}(\sigma(\Psi(x)))))) = \Phi(x), \text{ for } x\in L^p(M_2(\M),\varphi_2)= S^p_2(L^p(\M)).
\end{equation}
Since all maps on the left side of this equality are contractive, we obtain that the $n$-positive projection $\Phi : L^p(M_2(\M),\varphi_2)\to L^p(M_2(\M),\varphi_2)$ is contractive.

\end{proof}

\begin{proof}[Proof of Theorem \ref{main1}]
Consider the contractive and completely positive projection $$\Phi := \left[\begin{array}{cc}
    P_1 & P \\
    P^\circ & P_2
\end{array}\right] :S^p_2(L^p(\M,\varphi))\to S^p_2(L^p(\M,\varphi)) ,$$ with $P_1,P_2 : L^p(\M,\varphi)\to L^p(M,\varphi)$ some contractive projections. Their existence follows from Proposition \ref{Phi_proj} and Proposition \ref{phi_contr}. We use the same notation as in the previous proof. Then, we have a normal, contractive and completely positive projection $Q_{|\mathcal{N}} : \N\to \N$. By \cite[Proposition 4.1]{AR}, it is a conditional expectation. Denote by $\tilde{\N}$ its range, which is a von Neumann subalgebra of $\N$. Then, according to \cite[Example 5.8]{HJX}, the range of $Q_p$ is $L^p(\tilde{\N}, \psi_{|\tilde{\mathcal{N}}})$. We denote $\tilde{\phi}:= \psi_{|\tilde{\mathcal{N}}}$. Then, using the identity (\ref{Phi=}), we have \begin{equation*}
\begin{split}
    \Ran(\Phi) &= \sigma^{-1}\left(\kappa\left(Q_p\left( \kappa^{-1}\left(\sigma\left(\Psi\left(L^p(M_2(\M), \varphi_{2})\right)\right)\right)\right)\right)\right)\\ &= \sigma^{-1}\left(\kappa\left(Q_p\left( \kappa^{-1}\left(\sigma\left(eL^p(M_2(\M),\varphi_2)e\right)\right)\right)\right)\right)\\ &= \sigma^{-1}\left(\kappa\left(\Ran(Q_p)\right)\right)\\ &=  \sigma^{-1}\left(\kappa\left(L^p(\tilde{\N},\tilde{\phi})\right)\right).
\end{split}
\end{equation*} Note that the third equality hold because $$\sigma(eL^p(M_2(\M),\varphi_2)e) = e\sigma(L^p(M_2(\M),\varphi_2))e = eL^p(M_2(\M),\phi)e = L^p(\N,\phi_e),$$ and $\kappa^{-1}(L^p(\N,\phi_e)) = L^p(\N,\psi)$.

Set $e_1 := \left[\begin{array}{cc}
    \operatorname{s}(h) & 0 \\
    0 & 0
\end{array}\right]\in \N$, and set $e_2 := \left[\begin{array}{cc}
    0 & 0 \\
    0 & \operatorname{s}(k)
\end{array}\right] = 1_\mathcal{N} - e_1\in \N$.
Then, \[\Ran(P) \overset{\text{compl. isometric}}{\simeq} \left[\begin{array}{cc}
    0 & \Ran(P) \\
    0 & 0
\end{array}\right] \overset{(\ref{Phi_proj})}{=} \left[\begin{array}{cc}
    0 & \operatorname{s}(h)\Ran(P)\operatorname{s}(k) \\
    0 & 0
\end{array}\right] = e_1\Ran(\Phi)e_2.\]
 Since $e_1, e_2$ are in $\N$ and $\sigma$, $\kappa$ are the identity on $\N\subset M_2(\M)$ and are bimodular, we have \[\Ran(P) \simeq e_1\Ran(\Phi)e_2 = e_1\sigma^{-1}\left(\kappa\left(L^p(\tilde{\N},\tilde{\phi})\right)\right)e_2 = \sigma^{-1}\left(\kappa\left((e_1 L^p(\tilde{\N},\tilde{\phi})e_2\right)\right).\] Note that $$H^{1/2}Q(e_1)H^{1/2} = \Phi(H^{1/2}e_1H^{1/2}) = \Phi\left(\left[\begin{array}{cc}
    h  & 0 \\
    0  & 0
 \end{array}\right]\right)= \left[\begin{array}{cc}
    h  & 0 \\
    0  & 0
 \end{array}\right] = H^{1/2}e_1H^{1/2}.$$ We deduce that $e_1$ is in $\tilde{\N} = \Ran(Q).$ Similarly, $e_2\in \tilde{\N}.$ This ends the proof of Theorem \ref{main1}.
 \end{proof}
 
 \begin{remark}
 We can equip the space $L^p(\M,\varphi)$ with a bimodule structure, more precisely, with the same notation as above, it is a left $e_1\tilde{\N}e_1$-module and a right $(1-e_1)\tilde{\N}(1-e_1)$-module. Considering this bimodule structure, we can show that $P$ is a $e_1\tilde{\N}e_1 - (1-e_1)\tilde{\N}(1-e_1)$-bimodular map.
 
 We first highlight the bimodule structure of $L^p(\M,\varphi)$. We keep the same notation as above. Define the linear isometry \[\pi : x\in L^p(\M,\varphi)\mapsto \left[\begin{array}{cc}
    0  & x \\
    0  & 0
 \end{array}\right] \in S^p_2(L^p(\M,\varphi)).\] If $a\in e_1\tilde{\N}e_1$, we can write $a = \left[\begin{array}{cc}
    a'  & 0 \\
     0 & 0
 \end{array}\right]$, with $a'\in \operatorname{s}(h)\M \operatorname{s}(h)$. Then \[a\pi(x) = \left[\begin{array}{cc}
    a'  & 0 \\
     0 & 0
 \end{array}\right]\left[\begin{array}{cc}
    0  & x \\
    0  & 0
 \end{array}\right] = \left[\begin{array}{cc}
    0  & a'x \\
    0  & 0
 \end{array}\right] \in \Ran(\pi).\] It implies that $L^p(\M,\varphi)$ is a left $e_1\tilde{\N}e_1$-module with the bilinear contractive map $(a,x)\in e_1\tilde{\N}e_1\times L^p(\M,\varphi)\mapsto \pi^{-1}\left(a\pi(x)\right)$.
 
 Similarly, if $b\in (1-e_1)\tilde{\N}(1-e_1)$, we can write $b = \left[\begin{array}{cc}
    0  & 0 \\
    0  & b'
 \end{array}\right]$, with $b'\in \operatorname{s}(k)\M \operatorname{s}(k)$. Then \[b\pi(x) = \left[\begin{array}{cc}
    0  & 0 \\
     0 & b'
 \end{array}\right]\left[\begin{array}{cc}
    0  & x \\
    0  & 0
 \end{array}\right] = \left[\begin{array}{cc}
    0  & b'x \\
    0  & 0
 \end{array}\right] \in \Ran(\pi).\] We deduce that $L^p(\M,\varphi)$ is a right $(1-e_1)\tilde{\N}(1-e_1)$-module.
 
 Now, to see the bimodularity of $P$, note that for every $x\in L^p(\M,\varphi)$, \[\Phi\left(\pi(x)\right) = \Phi\left(\left[\begin{array}{cc}
    0  & x \\
    0  & 0
\end{array}\right]\right) = \left[\begin{array}{cc}
    0  & P(x) \\
    0  & 0
\end{array}\right] = \pi\left(P(x)\right).\] According to \cite[Theorem 1.1]{AR}, the map $\Phi$ is $\tilde{\N}$-bimodular.
Then, for $a = \left[\begin{array}{cc}
    a'  & 0 \\
     0 & 0
 \end{array}\right] \in e_1\tilde{\N}e_1$, $b = \left[\begin{array}{cc}
    0  & 0 \\
    0  & b'
 \end{array}\right] \in (1-e_1)\tilde{\N}(1-e_1)$, and $x\in L^p(\M,\varphi)$ we have \[\Phi(a\pi(x)b) = a\Phi(\pi(x))b = a\pi(P(x))b = \left[\begin{array}{cc}
    0  & a'P(x)b' \\
    0  & 0
 \end{array}\right],\] and \[\Phi(a\pi(x)b) = \Phi\left(\left[\begin{array}{cc}
    0  & a'xb' \\
    0  & 0
 \end{array}\right]\right) = \left[\begin{array}{cc}
    0  & P(a'xb') \\
    0  & 0
 \end{array}\right].\] We deduce that \(P(a'xb') = a'P(x)b'\), meaning that the projection $P$ is a $e_1\tilde{\N}e_1-(1-e_1)\tilde{\N}(1-e_1)$-bimodular map.
 
 \end{remark}

 
\begin{corollary}
Let $1<p<\infty$, let $\M$ be a $\sigma$-finite von Neumann algebra equipped with a normal faithful state. Let $P : L^p(\M,\varphi)\to L^p(\M,\varphi)$ be a bounded projection. Then $P$ is contractively $n$-pseudo decomposable for $n\geq 2$ if and only if $P$ is contractively decomposable.
\end{corollary}
\begin{proof}
We just need to prove the direct implication. If $P$ is contractively $n$-pseudo decomposable, for $n\geq 2$, by Proposition \ref{phi_contr_Lp}, we have a contractive $n$-positive projection $\Phi:= \left[\begin{array}{cc}
   P_1  & P \\
    P^\circ & P_2
\end{array}\right] :S^p_2\left(L^p(\M,\varphi)\right) \to S^p_2\left(L^p(\M,\varphi)\right)$, with $P_1,P_2:L^p(\M,\varphi)\to L^p(\M,\varphi)$ some contractive $n$-positive projections. Then according to \cite[Theorem 1.1]{AR}, $\Phi$ is a completely positive map, hence $P$ is contractively decomposable.
\end{proof}


\section{Contractive projections in Schatten spaces}\label{section dec Sp}
\subsection{Preliminaries}
We assume that the reader is familiar with Schatten spaces $S^p(H,K)$, for which we refer e.g. to \cite{McCarthy}. We only consider separable Hilbert spaces.

We use a natural embedding for tensor products. Let $H,$ $K$, $H'$ and $K'$ be Hilbert spaces. Then for $1\leq p<\infty$, we have \[S^p(H,K)\otimes S^p(H',K') \subset S^p\left(H\overset{2}{\otimes}H',  K\overset{2}{\otimes}K'\right)\] where $\overset{2}{\otimes}$ denotes the Hilbertian tensor product. The left-hand side of this inclusion is a dense subspace of the right-hand side. For any $a\in S^p(H,K),~b\in S^p(H',K')$, $\|a\otimes b\|_p = \|a\|_p\|b\|_p$.

For $1\leq p<\infty$, a set $A$ and a family of Banach spaces $(X_\alpha)_{\alpha\in A}$, we define the $p-$direct sum \[\bigoplus^p_{\alpha \in A} X_\alpha := \{(x_\alpha)_{\alpha \in A},\quad x_\alpha\in X_\alpha,~\sum_{\alpha \in A} \|x_\alpha\|^p<\infty\},\] endowed with the norm \[\|(x_\alpha)_{\alpha}\| := \left(\sum_{\alpha \in A} \|x_\alpha\|^p\right)^{\frac{1}{p}}.\] This is a Banach space.

For $(H_\alpha)_{\alpha \in A}$, and $(K_\alpha)_{\alpha \in A}$ families of Hilbert spaces, we have the embedding \[ \bigoplus^p_{\alpha \in A} S^p(H_\alpha,K_\alpha) \subset S^p\left(\bigoplus^2_{\alpha \in A} H_\alpha, \bigoplus^2_{\alpha \in A} K_\alpha\right),\] defining an element $x = (x_\alpha)_{\alpha}\in\bigoplus\limits^p_{\alpha \in A} S^p(H_\alpha,K_\alpha) $ as \[x((h_\alpha)_{\alpha}) := (x_\alpha(h_\alpha))_{\alpha},\quad \text{for }(h_\alpha)_{\alpha}\in \bigoplus\limits^2_{\alpha \in A} H_\alpha.\] Observe that this embedding determines a positive cone on \(\bigoplus\limits^p_{\alpha \in A} S^p(H_\alpha)\). We note that an element $(x_\alpha)_\alpha \in \bigoplus\limits^p_{\alpha \in A} S^p(H_\alpha)$ is positive if and only if each $x_\alpha$ is positive.
\subsubsection{Decomposable maps in $S^p$}

Let $H,K$ be Hilbert spaces, and let $1\leq p\leq \infty$. For $x$ an element of $S^p(H\overset{2}{\oplus}K)$, we can write $$x = \left[\begin{array}{cc}
   x_1  &  x_2\\
    x_3 & x_4
\end{array}\right],$$ with $x_1\in S^p(H)$, $x_2\in S^p(K,H)$, $x_3\in S^p(H,K)$ and $x_4\in S^p(K)$. If $s_H : H\overset{2}{\oplus}K\to H$ is the orthogonal projection onto $H$, and $s_K : H\overset{2}{\oplus}K\to K$ the orthogonal projection onto $K$, then $x_1 = s_Hxs_H^*$, $x_2 = s_Hxs_K^*$, $x_3 = s_Kxs_H^*$ and $x_4 = s_Kxs_K^*$. Then we have the canonical embedding $$S^p(K,H)\hookrightarrow S^p(H\overset{2}{\oplus}K),\quad x\mapsto \left[\begin{array}{cc}
    0 & x \\
    0 & 0
\end{array}\right] = s_H^*xs_K.$$

For $H,K,H',K'$ some Hilbert spaces, and for bounded maps $T_1 : S^p(H)\to S^p(H')$, $T_2: S^p(K,H)\to S^p(H',K')$, $T_3 : S^p(H,K)\to S^p(H',K')$ and $T_4 : S^p(K)\to S^p(K')$, we can consider the bounded map $T_{1234}:= \left[\begin{array}{cc}
    T_1 & T_2 \\
    T_3 & T_4
\end{array}\right] : S^p(H\overset{2}{\oplus}K)\to S^p(H'\overset{2}{\oplus}K')$ defined by  \[T_{1234}\left(\left[\begin{array}{cc}
   x_1  &  x_2\\
    x_3 & x_4
\end{array}\right]\right) = \left[\begin{array}{cc}
   T_1(x_1)  & T_2(x_2) \\
   T_3(x_3)  & T_4(x_4)
\end{array}\right],\quad \text{for }\left[\begin{array}{cc}
   x_1  &  x_2\\
    x_3 & x_4
\end{array}\right]\in S^p(H\overset{2}{\oplus}K).\]
Then, for $T: S^p(K,H)\to S^p(K',H')$, we can examine operator the operator \begin{equation}\label{T_tilde}
    \tilde{T} := \left[\begin{array}{cc}
    0 & T \\
    0 & 0
\end{array}\right] : S^p(H\overset{2}{\oplus}K)\to S^p(H'\overset{2}{\oplus}K').
\end{equation} We can also write \[\tilde{T}(x) = s_{H'}^*T(s_Hxs_K^*)s_{K'},\quad x\in S^p(H\overset{2}{\oplus}K).\]

We now define a decomposable map in the setting of Schatten spaces $S^p(K,H)$. We will see that we can define it with two equivalent formulations.
\begin{definition}\label{dec_on_Sp}Let $H,K,H',K'$ be Hilbert spaces, let $1\leq p\leq \infty$. We say that the operator $T : S^p(K,H)\to S^p(K',H')$ is decomposable (resp. $n$-pseudo-decomposable for $1\leq n\leq \infty$) if there exist bounded linear maps $u_1 : S^p(H)\to S^p(H')$, $u_2 : S^p(K)\to S^p(K')$ such that the map \[\Phi := \left[\begin{array}{cc}
    u_1 & T \\
    T^\circ & u_2
\end{array}\right] : S^p(H\overset{2}{\oplus}K)\to S^p(H'\overset{2}{\oplus}K')\] is completely positive (resp. $n$-positive). In this case, set $\|T\|_\mathrm{dec} := \inf\max\{\|u_1\|,\|u_2\|\}$ where the infimum is taken on all the maps $u_1,u_2$ such that $\Phi$ is completely positive (resp. set $\|T\|_{n,\mathrm{dec}} := \inf\max\{\|u_1\|,\|u_2\|\}$ with the infimum taken on all the maps $u_1,u_2$ such that $\Phi$ is $n$-positive).
\end{definition}
\begin{remark}
If $K=H$, we recover the usual definition of a decomposable operator (resp. $n$-pseudo-decomposable operator).
\end{remark}

\begin{lemma}\label{T-T_tilde}
Let $H,K,H',K'$ be Hilbert spaces, let $1\leq p\leq \infty$. The operator $T : S^p(K,H)\to S^p(K',H')$ is decomposable (resp. $n$-pseudo-decomposable for $1\leq n\leq \infty$) in the sense of Definition \ref{dec_on_Sp} if and only if the operator $\tilde{T} :S^p(H\overset{2}{\oplus}K)\to S^p(H'\overset{2}{\oplus}K')$ defined as in (\ref{T_tilde}) is decomposable (resp. $n$-pseudo-decomposable) in the sense of Definition \ref{dec_on_Lp}.

In this case, we have $\|T\|_\mathrm{dec} = \|\tilde{T}\|_\mathrm{dec}$ (resp. $\|T\|_{n,\mathrm{dec}} = \|\tilde{T}\|_{n,dec}$.)
\end{lemma}

\begin{proof}We use the same notation as at the beginning of this subsection. The proof for the $n$-pseudo-decomposable case is similar to the decomposable case; we therefore focus on the decomposable case.
Let $T : S^p(K,H)\to S^p(K',H')$, and assume that $\tilde{T}:S^p(H\overset{2}{\oplus}H)\to S^p(H'\overset{2}{\oplus}K')$ is decomposable. Then there exist some bounded linear maps $v_1,v_2 : S^p(H\overset{2}{\oplus}K)\to S^p(H'\overset{2}{\oplus}K')$ such that the map \[\Phi := \left[\begin{array}{cc}
    v_1 & \tilde{T} \\
    \tilde{T}^\circ & v_2
\end{array}\right] : S^p_2(S^p(H\overset{2}{\oplus}K))\to S^p_2(S^p(H'\overset{2}{\oplus}K'))\] is completely positive. Set $S : (H\overset{2}{\oplus}K)\overset{2}{\oplus}(H\overset{2}{\oplus}K)\to H\overset{2}{\oplus}K,\quad (\xi_1,\xi_2)\mapsto (s_H(\xi_1),s_K(\xi_2))$ and $S' : (H'\overset{2}{\oplus}K')\overset{2}{\oplus}(H'\overset{2}{\oplus}K')\to H'\overset{2}{\oplus}K',\quad (\xi_1,\xi_2)\mapsto (s_{H'}(\xi_1),s_{K'}(\xi_2))$. 
 The maps $S$ and $S'$ may be viewed as $$S = \left[\begin{array}{cc}
    s_H & 0 \\
    0 & s_K
\end{array}\right]\quad\text{and}\quad S' = \left[\begin{array}{cc}
    s_{H'} & 0 \\
    0 & s_{K'}
\end{array}\right].$$ Then, define $\Psi : x\in S^p(H\overset{2}{\oplus}K)\mapsto S'\Phi(S^*xS){S'}^*\in S^p(H'\overset{2}{\oplus}K')$. Since $\Phi$ is completely positive, the map $\Psi$ is completely positive. Set $u_1 : x\in S^p(H)\mapsto s_{H'}v_1(s_H^*xs_H)s_{H'}^*\in S^p(H')$ and $u_2 : x\in S^p(K)\mapsto s_{K'}v_2(s_K^*xs_K)s_{K'}^* \in S^p(K')$. Then we have \[\Psi = \left[\begin{array}{cc}
    u_1 & T \\
    T^\circ & u_2
\end{array}\right],\] and \(\|u_1\|\leq \|v_1\|,~\|u_2\|\leq \|v_2\|.\) It implies that $T$ is decomposbale, and $\|T\|_\mathrm{dec}\leq \|\tilde{T}\|_\mathrm{dec}$.

Conversely, assume $T: S^p(K,H)\to S^p(K',H')$ is decomposable. Then we have bounded linear maps $u_1 : S^p(H)\to S^p(H')$, $u_2 : S^p(K)\to S^p(K')$ such that the map \[\Psi := \left[\begin{array}{cc}
    u_1 & T \\
    T^\circ & u_2
\end{array}\right] : S^p(H\overset{2}{\oplus}K)\to S^p(H'\overset{2}{\oplus}K')\] is completely positive. Define $\Phi :  S^p_2(S^p(H\overset{2}{\oplus}K))\to S^p_2(S^p(H'\overset{2}{\oplus}K'))$ by \[\Phi(x) = {S'}^*\Psi(SxS^*)S'.\] This is a completely positive map. Set $v_1 : x\in S^p(H\overset{2}{\oplus}K)\mapsto s_{H'}^*u_1(s_Hxs_H^*)s_{H'}\in S^p(H'\overset{2}{\oplus}K')$ and $v_2 : S^p(H\overset{2}{\oplus}K)\mapsto s_{K'}^*u_2(s_Kxs_K^*)s_{K'}\in S^p(H'\overset{2}{\oplus}K')$. Then \[\Phi= \left[\begin{array}{cc}
    v_1 & \tilde{T} \\
    \tilde{T}^\circ & v_2
\end{array}\right],\] with \(\|v_1\|\leq \|u_1\|,~\|v_2\|\leq \|u_2\|.\) Hence $\tilde{T}$ is decomposbale, and $\|\tilde{T}\|_\mathrm{dec}\leq \|T\|_\mathrm{dec}$.
\end{proof}

\begin{remark}\label{appli_cb}
With the same notation as before, note that $T: S^p(K,H)\to S^p(K',H')$ is a projection if and only if $\tilde{T}$ is a projection.
Moreover, $T$ is completely bounded if and only if $\tilde{T} : S^p(H\overset{2}{\oplus}K)\to S^p(H'\overset{2}{\oplus}K')$ is completely bounded. In this case, we have $\|T\|_{cb} = \|\tilde{T}\|_{cb}$.
Then, using \cite[Theorem 3.30]{Arh_Krieg}, we deduce that a contractively decomposable operator $T: S^p(K,H)\to S^p(K',H')$ is completely contractive. 
\end{remark}

\begin{remark}\label{phi_contr}
Note that if $p=2$, we can compute \[\left\|\left[\begin{array}{cc}
    a & b \\
    c & d
\end{array}\right]\right\|_{S^2(H\overset{2}{\oplus}K)}^2 = \|a\|^2_{S^2(H)} + \|b\|^2_{S^2(K,H)} + \|c\|^2_{S^2(K,H)} + \|d\|^2_{S^2(K)}.\] Hence if $T: S^2(K,H)\to S^2(K',H')$ is decomposable with \[\Phi := \left[\begin{array}{cc}
    v_1 & T \\
    T^\circ & v_2
\end{array}\right] : S^2(H\overset{2}{\oplus}K)\to S^2(H'\overset{2}{\oplus}K')\] completely positive, then \[\|\Phi\| = \max\{\|v_1\|,\|v_2\|,\|T\|,\|T^\circ\|\}.\] So if $T : S^2(K,H)\to S^2(K',H')$ is contractively decomposable, the map $\Phi$ is contractive. This result is not known if $1<p\neq 2<\infty$.
\end{remark}

The following result is a simple adaptation of \cite[Lemma 5.1]{A_dec} in the space $S^p(K,H)$.

\begin{lemma}\label{axb_dec}
Let $H, H',K,K'$ be Hilbert spaces, let $1\leq p\leq \infty$, and let $a\in B(H,H'), b\in B(K,K')$. The operator \[M_{a,b} : S^p(K',H)\to S^p(K,H'),\quad x\mapsto axb,\] is decomposable, and \[\|M_{a,b}\|_\mathrm{dec}\leq \|a\|_\infty\|b\|_\infty.\]
\end{lemma} 

\begin{lemma}\label{composi_n_dec} Let $1\leq p<\infty$, let $1\leq n\leq \infty$, let $H_1,H_2,H_3$ and $K_1,K_2,K_3$ be Hilbert spaces.
Let $T_1 : S^p(K_1,H_1)\to S^p(K_2,H_2)$, and let $T_2 : S^p(K_2,H_2)\to S^p(K_3,H_3)$ be $n$-pseudo decomposable maps. Then $T_2\circ T_1$ is $n$-pseudo decomposable, with $$\|T_2\circ T_1\|_{n,\mathrm{dec}}\leq \|T_2\|_{n,\mathrm{dec}}\|T_1\|_{n,\mathrm{dec}}.$$
\end{lemma}

\subsubsection{Contractive projections in $S^p$}
We recall here some notions already used in \cite{AF1}, \cite{B} and \cite{MRR}.
\begin{definition}\label{op disjoint} Let $H,K$ be Hilbert spaces, let $1\leq p \leq \infty$.
Let $x,y\in S^p(K,H)$. We say that $x$ and $y$ are disjoint operators if \[x^*y = 0 \quad \mbox{ and } \quad xy^* = 0.\]

Two subspaces $X_1,X_2$ of the space $S^p(K,H)$ are called operator-disjoint if every pair $(x_1,x_2)\in X_1\times X_2$, consists of disjoint  operators.

A subspace $X$ of the space $S^p(K,H)$ is said to be indecomposable if it cannot be written as the sum of two nontrivial operator-disjoint subspaces. 
\end{definition}

\begin{remark}If $x,y\in S^p(K,H)$ are disjoint, then \begin{equation}\label{relation_norm_disjoint}
    \|x+y\|_p^p = \|x\|^p_p + \|y\|^p_p.
\end{equation}%

If $X_1$ and $X_2$ are two subspaces of $S^p(K,H)$, for $1\leq p\neq 2<\infty$, such that the identity (\ref{relation_norm_disjoint}) holds for every par $(x,y)\in X_1\times X_2$, then these subspaces are operator-disjoint. This follows from the equality case in Clarkson's inequality, see \cite[Theorem 2.7]{McCarthy}.
Therefore, if two subspaces $X_1,X_2$ of the space $S^p(K,H)$ are operator-disjoint, the space $X_1 
+ X_2 \subset S^p(K,H)$ coincide with the space $X_1\overset{p}{\oplus}X_2$. Here, we shall not distinguish between the internal and external direct sum.
\end{remark}

For \( 1 < p < \infty \), it follows from \cite{AF2} that every non-trivial subspace \( X\) of the space \(S^p(H, K) \) admits a direct sum decomposition: \begin{equation}\label{dec_X}
    X = \bigoplus^p_{\alpha} X_\alpha,
\end{equation} where each \( X_\alpha \) is non-trivial, indecomposable and the subspaces \( X_\alpha \) are pairwise operator-disjoint. Note that such a decomposition is unique up to the ordering of the family $(X_\alpha)_\alpha$. Moreover, $X$ is 1-complemented if and only if every subspace $X_\alpha$ is 1-complemented. According to \cite[Lemma 4.4]{B}, $X$ is positively 1-complemented if and only if every $X_\alpha$ is positively 1-complemented. 
\begin{remark}
For $p=1$, the decomposition (\ref{dec_X}) does not hold for any subspace of $S^1(H)$. Nevertheless, if we assume that $X$ is a \textit{1-complemented} subspace of the space $S^1(H)$, according to \cite[Theorem 2.14]{AF1}, we can decompose $X$ as a sum of a family of indecomposable pairwise operator-disjoint subspaces of $S^1(H)$, these subspaces are described in \cite[Theorem 2.14]{AF1} and are the same as in the case $1<p\neq 2<\infty$.

Note also that according to a remark in \cite[p.36]{AF1}, after Theorem 2.1, in $S^1(K,H)$, if we have a 1-complemented subspace that is non-degenerate, then there exists a \textit{unique} contractive projection onto it.
\end{remark}

We introduce a notion of equivalent spaces which is used in \cite{B} and \cite{MRR} to describe the (positively) 1-complemented subspaces of $S^p$.

For a subspace \(X\) of the space \(S^p(H)\), we say that $X$ is $*$-invariant when for any $x\in X$, $x^*\in X$.
\begin{definition}
Let $H,H',K,K'$ be Hilbert spaces, let $X\subset S^p(K,H)$ and $Y\subset S^p(K',H')$ be two closed subspaces. They are equivalent, denoted by $X\sim Y$, if there exist two partial isometries $U:H\to H'$ and $V:K\to K'$ such that  \[X = U^*YV\quad\text{and}\quad Y=UXV^*.\]

If $X\subset S^p(H)$ and $Y\subset S^p(H')$ are two closed, $*$-invariant subspaces, they are positively equivalent, denoted by $X\underset{\geq 0}{\sim}Y$, if there exists a partial isometry $U:H\to H'$  such that \[X = U^*YU\quad\text{and}\quad Y=UXU^*.\]
\end{definition}
\begin{remark}\label{equiv_subspace_1compl}
\begin{enumerate}
    \item According to \cite[Lemma 2.10]{B}, in the definition of equivalent subspaces, if we have two subspaces $X\subset S^p(K,H)$ and $Y\subset S^p(K',H')$ such that $$X\sim Y,$$ then we can choose the partial isometries $U$ and $V$ such that \[U^*U = \operatorname{s}_\ell(X),~~UU^*=\operatorname{s}_\ell(Y),~~V^*V=\operatorname{s}_r(X),~~VV^*=\operatorname{s}_r(Y).\] Hence if the two subspaces $X$ and $Y$ are non-degenerate, then we can choose $U$ and $V$ unitaries. We have similar results in the case of positively equivalent spaces.
    \item If we have two equivalent subspaces $X\subset S^p(K,H)$ and $Y\subset S^p(K',H')$, with partial isometries $U:H\to H'$, $V:K\to K'$ such that $$X=U^*YV,\quad \text{ and }\quad Y=UXY^*,$$ then $X$ is 1-complemented if, and only if $Y$ is. In this case, if $P:S^p(K,H)\to S^p(K,H)$ denotes the contractive projection onto $X$, then the map \begin{equation}\label{passage_P_a_Q}
        Q : x\in S^p(K',H')\mapsto UP(U^*xV)V^* \in S^p(K',H')
    \end{equation} is the contractive projection onto $Y$.
    \item It follows from (\ref{passage_P_a_Q}) that if $X\subset S^p(H)$ and $Y\subset S^p(K)$ are positively equivalent, then $X$ is (completely) positively 1-complemented if and only if $Y$ is.
\end{enumerate}
\end{remark}
We recall the main result of \cite{B}: the description of positively 1-complemented subspaces of $S^p(H)$.
Let \( I \) be a countable index set. We denote by $S^p_{I}$ the space $S^p(\ell^2_I)$. Denote by $\top : S^p_{I} \to S^p_{I}$ the transpose map. Regarding elements of $S^p_{I}$ as scalar matrices, \[\top([t_{ij}]_{i,j\in I}) = [t_{ij}]^\top_{i,j\in I} := [t_{ji}]_{i,j\in I}.\]
 We define the spaces of symmetric and anti-symmetric matrices within \( S^p_I \), respectively, as follows \[\S^p_I := \{w\in S^p_I, \quad w^\top = w\} \quad \mbox{ and } \quad \A^p_I := \{w\in S^p_I, \quad w^\top = -w\}.\]

Let $N\geq 2$ be an integer, and let $\mathcal{H}=\ell^2_N$, with $(e_k)_{1\leq k\leq N}$ denoting the canonical basis. For any $m=0,\ldots ,N$, let $\H^{\wedge m}:= \H\wedge\cdots \wedge \H$ denote the $m$-fold anti-symmetric tensor product of $\H$. We define \( \Lambda_N \) as the Hilbertian direct sum of these tensor products:
\[
\Lambda_N := \bigoplus_{0\leq m\leq N}^{2} \H^{\wedge m}.
\]
This space is known as the Fock space associated with \( \H \). Note that for $0\leq m\leq N$, \[\dim(\H^{\wedge m}) = \binom{N}{m},\quad \text{ and }\quad \dim(\La_N) = 2^N.\]
For $1\leq k\leq N$, the creation operator $c_k$ is defined by
 \[ c_k: \left|\begin{array}{ccc}
    \La_N & \to & \La_N \\
    x & \mapsto & e_k\wedge x
\end{array}\right. .\]

\begin{definition}\label{def_spin} Let \( N \in \mathbb N^* \), $H$ a Hilbert space and let \((w_1, \ldots, w_N)\) be a family of operators in $B(H)$. This family is a spin system in $B(H)$ if \begin{itemize}
    \item \(w_j\) is a self-adjoint unitary operator, for every $j\in \{1,\ldots,N\}$,
    \item $w_iw_j + w_jw_i = 0,$ for every \(1\leq i\neq j\leq N\)
\end{itemize}
\end{definition}

We now define a particular spin system. Let \( N \in \mathbb N^* \), and for each \( 1 \leq k \leq N \), we define
\begin{equation}\label{defs_j}
    s_k := c_k + c_k^* \quad \text{and} \quad s_{-k} := \frac{c_k - c_k^*}{i}.
\end{equation}
It is straightforward to see that the family \( (s_1, s_2, \ldots, s_N, s_{-1}, \ldots, s_{-N}) \) forms a spin system in \( B(\Lambda_N) \). For each nonempty subset \( A \subset \{-N, \ldots, -1,1, \ldots,  N\} \), let $k=|A|$ denote its cardinality, and write \[ A = \{i_1 , i_2 , \ldots , i_k\}, \] with $i_1 < i_2 < \ldots < i_k$. We set
\[
s_A := s_{i_1}s_{i_2}\ldots s_{i_k}.
\] Then if \( A = \emptyset \) , we set $$s_A = 1.$$
The family \( \{s_A \mid A \subset \{-N, \ldots, -1,1, \ldots,  N\}\} \) forms a basis for the \( C^* \)-algebra \[C^*\langle s_1, \ldots, s_N, s_{-1}, \ldots, s_{-N} \rangle,\] generated by the spin system. Since both sides have the same dimension, we have
\[
B(\Lambda_N) = C^*\langle s_1, \ldots, s_N, s_{-1}, \ldots, s_{-N} \rangle.
\]
We define the space \(\mathcal{F}_N\) as
\begin{equation}\label{defFn}
    \mathcal{F}_N := \text{span}\left\{1, s_1, \ldots, s_N, s_{-1}, \ldots, s_{-N}, s_{-N}s_N \ldots s_{-1}s_1\right\} \subset B(\Lambda_N),
\end{equation}
and we define \(\mathcal{E}_{2N}\) as \begin{equation}\label{defE2N}
    \mathcal{E}_{2N} := \text{span}\{1,s_1,\ldots,s_N,s_{-1},\ldots,s_{-N}\}\subset B(\Lambda_N).
\end{equation}
Note that the family $(s_1,\ldots, s_N, s_{-1},\ldots, s_{-N}, i^Ns_{-N}s_N\cdots s_{-1}s_1)$ is also a spin system in $B(\La_N)$.
Next, we introduce the linear map \(\sigma : \F_N\to\F_N\) defined by 
\begin{equation}\label{defsigma}
   \sigma(1)=1,\quad \sigma(s_j) = s_j,\quad \sigma(s_{-N}s_{N}\cdots s_{-1}s_1) = - s_{-N}s_{N}\cdots s_{-1}s_1,
\end{equation}
 for $j\in\{-N,\ldots, -1,1,\ldots, N\}.$
This map is an involution. Moreover, this is an isometry from $\F_N^p$ to itself, where $\F_N^p$ is the subspace $\F_N$ endowed with the $S^p$-norm, for $p\geq 1$. We also denote $\E^p_{2N}$ the subspace $\E_{2N}$ endowed with the $S^p$-norm.

\begin{theorem}\label{main_B}
Let $1\leq p <\infty$, let $H$ be a Hilbert space, and let $X$ be an indecomposable subspace of $S^p(H)$. Then $X$ is positively 1-complemented in $S^p(H)$ if, and only if, $X$ is positively equivalent to one of the following spaces : 
\begin{enumerate}
    \item The space $O\S^p_I \otimes a$, where $O\in \S_I$ is a symmetric unitary operator.
    
    \item The space $O\A^p_I \otimes a$, where $O\in \A_I$ is an anti-symmetric unitary operator.
    
    \item The space $\{(w\otimes a_1,w^\top\otimes a_2),\quad w\in S^p_I\}$.
    
    \item The space $v\E_{2N}^p\otimes a$, where $v\in\E_{2N}$ is a unitary operator.
    
    \item The space $\{(vx\otimes a_1, \sigma(v)\sigma(x)\otimes a_2),\quad x\in \F_N^p\}$, where $v\in\F_N$ is a unitary operator,
\end{enumerate}
with $a,a_1\in S^p(H_1)$, $a_2\in S^p(H_2)$ some positive injective operators on Hilbert spaces $H_1,~H_2$, $I$ a countable set, and $N\geq 2$.
\end{theorem}

\begin{remark}
Note that in the previous theorem, the subspaces $\S_I$, $\A_I$, $\F_N$ and $\E_{2N}$ are $JC^*$-triple. A $JC^*$-triple is a subspace $J$ of $B(H)$, for $H$ a Hilbert space, such that for every $x,y,z\in J$, $\frac{1}{2}(xy^*z + zy^*x) \in J$. Note also that a unital $JC^*$-triple (as $\S_I$, $\F_N$ and $\E_{2N}$) is $*$-invariant.  For more information about this notion, we refer to \cite[Section 3]{B}, \cite{H2}, and \cite{H}.
\end{remark}

\subsection{Completely positive contractive projection in $S^p$}

\begin{lemma}\label{dec_space} Let $1\leq p<\infty$,  let $H$ be a Hilbert space, and let $X$ be a subspace of $S^p(H)$. Consider a decomposition $$X = \bigoplus_\alpha^p X_\alpha,$$ with a family of indecomposable pairwise operator-disjoint subspace $X_\alpha \subset S^p(H)$. Then $X$ is completely positively 1-complemented if and only if every $X_\alpha$ is completely positively 1-complemented.
\end{lemma} 
 \begin{proof} The proof is the same as for \cite[Lemma 4.4]{B}, noting that for $s\in B(K,H)$, the map $x\in S^p(H)\mapsto sxs^* \in S^p(K)$ is completely positive. \end{proof}
It follows that to obtain a description of completely positively 1-complemented subspaces of $S^p(H)$, we only need to describe all completely positively 1-complemented and indecomposable subspaces.

\begin{proposition}\label{projcp} Let $1\leq p <\infty$, let $X$ be an indecomposable subspace of $S^p(H)$. Then $X$ is completely positively 1-complemented if, and only if there exist an index set $I$, $a\in S^p(K)$ a positive and injective operator on a Hilbert space $K$, and $U : \ell^2_I(K)\to H$ an isometry such that $$X = U(S^p_I\otimes a)U^*.$$
\end{proposition}
\begin{proof}
To prove the "if" direction, by point (3) of Remark \ref{equiv_subspace_1compl}, it suffices to show that for every index set $I$ and every positive injective operator $a\in S^p(K)$ on a Hilbert $K$ space with $\|a\|_p=1$, the subspace $$X = S^p_I\otimes a\subset S^p(\ell^2_I(K)),$$ is completely positively 1-complemented. This subspace is non-degenerate, so for every $1\leq p<\infty$ (even $p=1$), there is a unique contractive projection onto $S^p_I\otimes a$. Let $P : S^p(\ell^2_I(K))\to S^p(\ell^2_I(K))$ be this contractive projection. Define $\varphi_a :S^p(K)\to S^p(K)$ by \(\varphi_a(x) = \Tr(a^{p-1}x)a.\) Using \cite[Lemma 2.1]{B}, it is straightforward to see  that $\varphi_a$ is completely positive and completely contractive. Then \[P = \Id_{S^p_I}\otimes\varphi_a,\] and we deduce that $P$ is completely positive. 

For the converse, consider $X$ a completely positively 1-complemented, indecomposable subspace of $S^p(H)$, and let $P : S^p(H)\to S^p(H)$ be the contractive and completely positive projection onto $X$. We may assume that $X$ is non-degenerate, that is $\operatorname{s}(P) = 1$, where $\operatorname{s}(P)$ is the support projection of $P$. Indeed, if $X$ is degenerate, let $\tilde{H} = \operatorname{s}(P)(H)$, and let $s: H\to \tilde{H}$ be the orthogonal projection onto $\tilde{H}$. Then $X$ is positively equivalent to the subspace $sXs* \subset S^p(\tilde{H})$ which is non-degenerate, and completely positively 1-complemented, according to point (3) of Remark \ref{equiv_subspace_1compl}.

We first assume that $p\neq 2$. According to \cite[Proposition 3.31]{Arh_Krieg}, $P$ is completely contractive. 
We apply \cite[Theorem 1.1]{MRR}. There exist index sets $I, J$, $a\in S^p(K)$ an operator on Hilbert space $K$ and two isometries $V : \ell^2_I(K)\to H$, $U : \ell^2_J(K)\to H$ such that \[X = V(S^p_{I,J}\otimes a)U^*.\] Using \cite[Lemma 2.12]{B}, we may assume that $a$ is a positive and injective operator. By the point (1) of Remark \ref{equiv_subspace_1compl}, we may assume that $U$ and $V$ are unitaries. Applying Theorem \ref{main_B} (or studying its proof in \cite{B} for the case $X = V(S^p_{I,J}\otimes a)U^*$ ), we obtain that \[X = U(S^p_J\otimes a)U^*.\] (There exists a unitary $u:\ell^2_J\to \ell^2_I$ such that $U = V(u\otimes \Id_K)$.) Hence we proved the result for $1\leq p\neq 2<\infty$.

For the case $p=2$, we heavily rely on \cite[Section 6]{B}. Fix $h\in X\subset S^2(H)$ a positive injective element, with $\|h\|_2 = 1$. The existence of $h$ is ensured by Proposition \ref{existence_h}. Then according to \cite[Proposition 6.2]{B}, the operator $V_1 : hB(H)h\to S^1(H)$ defined by \[V_1(hxh) = h^{1/2}P(h^{1/2}xh^{1/2})h^{1/2}\] is well defined and extends to a contractive and positive projection on $S^1(H)$. For every $n\in\Ndb$, and $y\in M_n(B(H))$, we have \[\left[\Id_{S^1_n}\otimes V_1\right]\left((I_n\otimes h)y(I_n\otimes h)\right) = (I_n\otimes h)^{1/2}[\Id_{S^2_n}\otimes P]\left((I_n\otimes h)^{1/2}y(I_n\otimes h)^{1/2}\right)(I_n\otimes h)^{1/2}.\] Since $P$ is completely positive, using the same argument as for the positivity of $V_1$, we obtain that $\Id_{S^1_n}\otimes V_1$ is positive, hence $V_1$ is completely positive. Applying \cite[Proposition 6.3]{B}, we define a positive contractive projection $V_p : S^p(H)\to S^p(H)$ for $1\leq p\leq 2$, such that \[V_p(h^{1/p}xh^{1/p}) = h^{\frac{1}{p}-\frac{1}{2}}P(h^{1/2}xh^{1/2})h^{\frac{1}{p}-\frac{1}{2}},\quad x\in B(H). \] By the same argument as for $V_1$, the map $V_p$  is completely positive. For $2\leq q\leq \infty$, set $V_q := V_p^*$, for $\frac{1}{p} + \frac{1}{q} =1$. It is a completely positive and contractive projection.
Fix $1<p<\infty$, and $2<q<\infty$ such that $\frac{1}{p}+\frac{1}{q}=1$. Set $X_p := \Ran(V_p)\subset S^p(H)$, $X_q := \Ran(V_q)\subset S^q(H)$. By \cite[Proposition 6.6]{B}, $X_p$ and $X_q$ are non-degenerate and indecomposable. By applying the present proposition in the case $q\neq 2$, we have $X_q = U(S^q_I\otimes a)U^*$, for $I$ an index set, $a\in S^q(K)$ a positive injective operator on a Hilbert space $K$, and $U : \ell^2_I(K)\to H$ a unitary operator. By \cite[Proposition 6.9]{B} with $J = B(H)$, we have \[X_p = U(S^p_I\otimes a^{q-1})U^*\quad \text{and}\quad X = U(S^2_I\otimes a^{q/2})U^*.\] Hence the result follows.
\end{proof}

\subsection{Description of contractively 1-pseudo decomposable projections}


\begin{proposition}\label{equiv_subs}Let $H,H',K,K'$ be Hilbert spaces, let $1\leq n\leq \infty$ and let $1\leq p<\infty$.  Let $X\subset S^p(K,H)$ and $Y\subset S^p(K',H')$ be two equivalent subspaces. Then $X$ is 1-complemented with a contractively $n$-pseudo decomposable projection if and only if $Y$ has the same property.
\end{proposition}
\begin{proof}
Assume that $X$ is 1-complemented, with $P$ the contractively $n$-pseudo decomposable projection onto it. We have two partial isometries $U:H\to H'$ and $V:K\to K'$ such that \[X = U^*YV\quad\text{ and }\quad Y=UXV^*.\] Denote by $Q : S^p(K',H')\to S^p(K',H')$ the contractive projection onto $Y$. Then, according to the point (1) of Remark \ref{passage_P_a_Q}, with notation of Lemma \ref{axb_dec}, we have \[Q = M_{U,V^*}\circ P\circ M_{U^*,V}.\]
Since $U$ and $V$ are contractive, we deduce by Lemma \ref{axb_dec} that $M_{U,V^*}$ and $M_{U^*,V}$ are contractively decomposable, hence contractively $n$-pseudo decomposable. Then, with Lemma \ref{composi_n_dec}, we obtain that $Q$ contractively $n$-pseudo decomposable.
\end{proof}

\begin{remark}\label{nondegen_subs}
Let $X$ be a subspace of $S^p(K,H)$. Let $H':=\Ran(\operatorname{s}_\ell(X))$, let $K':=\Ran(\operatorname{s}_r(X))$, and let $\operatorname{s}_\ell:H\to H'$ and $\operatorname{s}_r:K\to K'$ be the orthogonal projection onto $H'$ and $K'$ respectively. Then $X$ is equivalent to the subspace $\operatorname{s}_\ell X \operatorname{s}_r^* \subset S^p(K',H')$, which is non-degenerate. Then, according to Proposition \ref{equiv_subs}, we only need to study the non-degenerate subspaces that are the range of a contractively $n$-pseudo decomposable projection.
\end{remark}

The next proposition is the part $(3)\Rightarrow (2)$ of Theorem \ref{main2}.
\begin{proposition}
    Let $H,K$ be Hilbert spaces, let $1\leq p<\infty$, let $X$ be a subspace of $S^p(K,H)$.
   Assume that $X$ is equivalent to a subspace of the form \[\bigoplus_{\alpha\in A}^p S^p_{I_\alpha, J_\alpha}\otimes a_\alpha,\] where $A$ is a set,  $(I_\alpha)_{\alpha\in A}$ and $(J_\alpha)_{\alpha \in A}$ two families of indices, for every $\alpha\in A$, $a_\alpha \in S^p(K_\alpha)$ is a positive and injective operator on a Hilbert spaces $K_\alpha$. 
   Then $X$ is 1-complemented with a contractively decomposable projection.
\end{proposition}

\begin{proof}
According to Proposition \ref{equiv_subs} and Remark \ref{nondegen_subs}, we only need to consider \[X = \bigoplus_{\alpha\in A}^p \left(S^p_{I_\alpha,J_\alpha}\otimes a_\alpha\right)\subset S^p\left( \bigoplus^2_{\alpha \in A} \ell^2_{J_\alpha}(K_\alpha), \bigoplus^2_{\alpha \in A} \ell^2_{I_\alpha}(K_\alpha)\right).\]
In this case, we have $K = \underset{\alpha}{\overset{2}{\bigoplus}} \ell^2_{J_\alpha}(K_\alpha)$ and $H =\underset{\alpha}{\overset{2}{\bigoplus}} \ell^2_{I_\alpha}(K_\alpha).$
According to \cite[Main Theorem]{AF2}, it is a 1-complemented subspace. Denote by $P$ the contractive projection from $S^p(K,H)$ onto $X$.
Set \[\H := \left(\bigoplus^2_\alpha \ell^2_{I_\alpha}(K_\alpha)\right)\overset{2}{\oplus}\left(\bigoplus^2_\alpha \ell^2_{J_\alpha}(K_\alpha)\right) = \bigoplus^2_\alpha \left(\ell^2(K_\alpha)\overset{2}{\oplus}\ell^2_{J_\alpha}(K_\alpha)\right) = \bigoplus^2_\alpha \ell^2_{I_\alpha\sqcup
 J_\alpha}(K_\alpha), \] where $I_\alpha\sqcup J_\alpha$ denotes the index set resulting from the concatenation of $I_\alpha$ and $J_\alpha$. Denote by\[Q : S^p(\H)\to S^p(\H)\] the contractive projection onto \(\underset{\alpha}{\overset{p}{\bigoplus}} S^p_{I_\alpha\sqcup J_\alpha}\otimes a_\alpha.\)
Combining Lemma \ref{dec_space} and Proposition \ref{projcp}, we obtain that $Q$ is completely positive. Using the decomposition \(\H = H\overset{2}{\oplus}K\), we can write $Q$ as \[Q = \left[\begin{array}{cc}
    P_1 & P \\
    P^\circ & P_2
\end{array}\right],\] where $$P_1 : S^p\left(\bigoplus^2_\alpha \ell^2_{I_\alpha}(K_\alpha)\right)\to S^p\left(\bigoplus^2_\alpha \ell^2_{I_\alpha}(K_\alpha)\right)$$ is the contractive projection onto \(\underset{\alpha}{\overset{p}{\bigoplus}} S^p_{I_\alpha}\otimes a_\alpha\), and $$P_2 : S^p\left(\bigoplus^2_\alpha \ell^2_{J_\alpha}(K_\alpha)\right)\to S^p\left(\bigoplus^2_\alpha \ell^2_{J_\alpha}(K_\alpha)\right)$$ is the contractive projection onto \(\underset{\alpha}{\overset{p}{\bigoplus}} S^p_{J_\alpha}\otimes a_\alpha\). We deduce that $P$ is contractively decomposable.
\end{proof}

The remainder of this subsection is devoted to the part $(1)\Rightarrow(3)$ of Theorem \ref{main2}. The idea is to consider $P:S^p(K,H)\to S^p(K,H)$ a contractively 1-pseudo decomposable projection, then to obtain a contractive and positive projection $\Phi$ that could be written as $$\Phi := \left[\begin{array}{cc}
    P_1 & P \\
    P^\circ & P_2
\end{array}\right] : S^p(H\overset{2}{\oplus}K)\to S^p(H\overset{2}{\oplus}K)$$ with $P_1 : S^p(H)\to S^p(H)$, $P_2:S^p(K)\to S^p(K)$ some non-zero positive and contractive projections. We say that $\Phi$ is a \textit{$M_2$-splitting} positive contractive projection. Then, using Theorem \ref{main_B}, we have the description of the range of a positive contractive projection on $S^p(H\overset{2}{\oplus}K)$. Adding the assumption to be $M_2$-splitting, we show that, in this case, the range of such a projection is of the form $U\left(\underset{\alpha}{\overset{p}{\bigoplus}}S^p_{I_\alpha}\otimes a_\alpha\right)U^*$, with $a_\alpha\in S^p(H_\alpha)$ a positive injective operator on a Hilbert space $H_\alpha$, an index set $I_\alpha$, for every $\alpha$, and a partial isometry $U : \underset{\alpha}{\overset{2}{\oplus}}\ell^2_{I_\alpha}(H_\alpha)\to H\overset{2}{\oplus}K$. This implies that $\Phi$ is completely positive, hence $P$ is contractively decomposable. Then, we deduce the description of its range as stated in Theorem \ref{main2}.

The following is a reformulation of Propositions \ref{Phi_proj} and \ref{phi_contr} in the setting of Schatten spaces $S^p(K,H)$. 
\begin{lemma}\label{phi_proj_Sp}Let $1<p<\infty$, let $H,K$ be Hilbert spaces.
Let $P:S^p(K,H)\to S^p(K,H)$ be a contractively 1-pseudo decomposable projection, with a non-degenerate range. Then, there exist positive contractive projections $P_1 : S^p(H)\to S^p(H)$, $P_2:S^p(K)\to S^p(K)$ with non-degenerate ranges such that the linear map $$\Phi := \left[\begin{array}{cc}
    P_1 & P \\
    P^\circ & P_2
\end{array}\right] : S^p(H\overset{2}{\oplus}K)\to S^p(H\overset{2}{\oplus}K)$$ is a contractive and positive projection.
\end{lemma}
\begin{proof}
The proof of the existence of $P_1$ and $P_2$ is obtained with an easy adaptation of the proof of \cite[Proposition 7.1]{A_dec}. 

To show that $\Phi$ is contractive, it suffices to apply Proposition \ref{phi_contr} with the von Neumann algebra $\M= B(H\overset{2}{\oplus}K)$ and to use the arguments in the proof of Lemma \ref{T-T_tilde}.
\end{proof}

\begin{lemma}\label{corner_preserv}
Let $\H$ be a Hilbert space, let $1<p<\infty$, and let $\Phi : S^p(\H)\to S^p(\H)$ be a contractive and positive projection. Then the following two assertions are equivalent :
\begin{enumerate}
    \item There exist a non-trivial orthogonal decomposition of $\H$, $\H = H\overset{2}{\oplus}K$ and non-zero contractive projections $P_1 : S^p(H)\to S^p(H)$, $P_2:S^p(K)\to S^p(K)$, $P : S^p(K,H)\to S^p(K,H)$ such that $\Phi$ may be expressed as \[\Phi = \left[\begin{array}{cc}
    P_1 & P \\
    P^\circ & P_2
\end{array}\right] : S^p(H\overset{2}{\oplus}K)\to S^p(H\overset{2}{\oplus}K).\]
\item There exists an orthogonal projection $p\in B(\H)$ such that $p\neq 0,\Id_\mathcal{H}$ and $\Phi\circ L_p = L_p\circ \Phi$, where $L_p : x\mapsto px$.
\end{enumerate}
In this case, if the subspace $\Ran(P)\subset S^p(K,H)$ is non-degenerate, then $\Ran(P_1)\subset S^p(H)$ and $\Ran(P_2)\subset S^p(K)$ are also non-degenerate.
\end{lemma}

\begin{definition}Let $1<p<\infty$, and let $\H$ be a Hilbert space.
A contractive projection $\Phi : S^p(\H)\to S^p(\H)$ that satisfies one the two equivalent assertions of Lemma \ref{corner_preserv} is called $M_2$-splitting.
\end{definition}
\begin{proof}[Proof of Lemma \ref{corner_preserv}]
For the implication $(1)\Rightarrow (2)$, let $p\in B(\H)$ denote the orthogonal projection onto $H$. Then we can write $p = \left[\begin{array}{cc}
    \Id_H & 0 \\
    0 & 0
\end{array}\right] : B(H\overset{2}{\oplus}K)\to B(H\overset{2}{\oplus}K)$. We obtain, for $x = \left[\begin{array}{cc}
    x_1& x_2 \\
    x_3 & x_4
\end{array}\right] \in S^p(H\overset{2}{\oplus}K)$, \[L_p\circ\Phi(x) = \left[\begin{array}{cc}
    1& 0 \\
    0 & 0
\end{array}\right]\left[\begin{array}{cc}
    P_1(x_1)& P(x_2) \\
    P^\circ(x_3) & P_2(x_4)
\end{array}\right] = \left[\begin{array}{cc}
    P_1(x_1)& P(x_2) \\
    0 & 0
\end{array}\right]\] and \[\Phi\circ L_p(x) = \Phi\left(\left[\begin{array}{cc}
    x_1& x_2 \\
    0 & 0
\end{array}\right]\right) = \left[\begin{array}{cc}
    P_1(x_1)& P(x_2) \\
    0 & 0
\end{array}\right] = L_p\circ\Phi(x).\]

Conversely, assume that we have $0<p< \Id_\mathcal{H}$ an orthogonal projection such that \begin{equation}\label{calc1}
    \Phi\circ L_p = L_p\circ \Phi.
\end{equation} Then we also have \begin{equation}\label{calc2}
    \Phi\circ L_{\Id_\mathcal{H}-p} = L_{\Id_\mathcal{H}-p}\circ \Phi.
\end{equation} Since $\Phi$ is positive, it is adjoint preserving, we deduce that \begin{equation}\label{calc3}
     \Phi\circ R_p = R_p\circ \Phi \quad \text{and}\quad \Phi\circ R_{\Id_\mathcal{H}-p} = R_{\Id_\mathcal{H}-p}\circ \Phi,
\end{equation} where $R_p: x \mapsto xp$. Then, set $H := \Ran(p)$ and $K := H^\perp$, and denote by $s_1 : \H\to H$ and $s_2 : \H\to K$ the orthogonal projections onto $H$ and $K$ respectively, so that $p=s_1^*s_1$ and $\Id_\mathcal{H}-p = s_2^*s_2$. Define $P_1 := s_1\Phi(s_1^*\cdot s_1)s_1^* : S^p(H)\to S^p(H)$, $P_2 := s_2\Phi(s_2^*\cdot s_2)s_2^* : S^p(K)\to S^p(K)$ and $P := s_1\Phi(s_1^*\cdot s_2)s_2^*$.
Using (\ref{calc1}), (\ref{calc2}) and (\ref{calc3}), it is easy to show that $P$, $P_1$ and $P_2$ are contractive projections. Moreover, $\Phi = \left[\begin{array}{cc}
    P_1 & P \\
    P^\circ & P_2
\end{array}\right]$. Indeed, if $X = \left[\begin{array}{cc}
    x & 0 \\
    0 & 0
\end{array}\right]$, then $pX=X$, hence, by (\ref{calc1}) $p\Phi(X) = \Phi(X)$. It follows that $s_2\Phi(X)=0$, and therefore $s_2\Phi(X)s_1^*=0$. Similarly, using (\ref{calc3}), we obtain $s_1\Phi(X)s_2^*=0$ and $s_2\Phi(X)s_2^*=0$. Therefore $\Phi(X) = \left[\begin{array}{cc}
    P_1(x) & 0 \\
    0 & 0
\end{array}\right]$. The argument is similar for the other corners, hence for $X = \left[\begin{array}{cc}
    x_1 & x_2 \\
    x_3 & x_3
\end{array}\right]$, we have $\Phi(X) = \left[\begin{array}{cc}
    P_1(x_1) & P(x_2) \\
    P^\circ(x_3) & P_2(x_3)
\end{array}\right] $.

Furthermore, if $\Phi = \left[\begin{array}{cc}
    P_1 & P \\
    P^\circ & P_2
\end{array}\right]$ is a positive projection, then we have $s_\ell(P)\leq \operatorname{s}(P_1)$ and $s_r(P)\leq \operatorname{s}(P_2)$ as in Remark \ref{ran(P)nondegenimpliesran(Phi)nondegen}. If $\Ran(P)\subset S^p(K,H)$ is non-degenerate, we obtain $\operatorname{s}(P_1)=\Id_H$ and $\operatorname{s}(P_2)=\Id_K$.

\end{proof}

\begin{lemma}\label{phi_diago_de_cornerpreserving}Let $\H$ be a Hilbert space, let $1<p<\infty$, and let $\Phi : S^p(\H)\to S^p(\H)$ be a contractive and positive projection. We assume that $\Phi$ is $M_2$-splitting, that is, with notation of Lemma \ref{corner_preserv}, \[\Phi = \left[\begin{array}{cc}
    P_1 & P \\
    P^\circ & P_2
\end{array}\right] : S^p(H\overset{2}{\oplus}K)\to S^p(H\overset{2}{\oplus}K),\] with $\H = H\overset{2}{\oplus}K$ and non-zero contractive projections $P_1 : S^p(H)\to S^p(H)$, $P_2:S^p(K)\to S^p(K)$, $P : S^p(K,H)\to S^p(K,H)$. We assume, in addition that $\Ran(P)\subset S^p(K,H)$ is non-degenerate.
If the range of $\Phi$ may be written as a sum of indecomposable, non-trivial and pairwise operator-disjoint subspaces $X_\alpha$, that is $$\Ran(\Phi) = \bigoplus^p_{\alpha\in A} X_\alpha,$$ then each $X_\alpha$ is positively 1-complemented with a $M_2$-splitting contractive projection.
\end{lemma}
\begin{proof}Note that if the set $A$ is a singleton, the result is immediate. Thus, we assume that $A$ contains at least two elements.  According to \cite[Lemma 4.4]{B}, each $X_\alpha$ is positively 1-complemented.
Let $\H_\alpha := \Ran(\operatorname{s}(X_\alpha))$, and denote by $\Phi_\alpha : S^p(\H_\alpha)\to S^p(\H_\alpha)$ the contractive projection onto $X_\alpha \subset S^p(\H_\alpha)$. Note that the subspaces $\H_\alpha\subset \H$ are pairwise orthogonal. Then, if $x = (x_\alpha)_\alpha\in \overset{p}{\underset{\alpha}{\bigoplus}} S^p(\H_\alpha)$, \[\Phi(x) = (\Phi_\alpha(x_\alpha))_\alpha.\]
According to Lemma \ref{corner_preserv}, the subspaces $\Ran(P_1)\subset S^p(H)$ and $\Ran(P_2)\subset S^p(K)$ are non-degenerate. Applying Proposition \ref{existence_h}, there exists a positive element $h\in\Ran(P_1)$ such that $\operatorname{s}(h)=\Id_H$. Then, the positive element $x := \left[\begin{array}{cc}
    h & 0 \\
     0& 0
\end{array}\right]$ is in the range of $\Phi$, and $\operatorname{s}(x) = p_H$, the orthogonal projection onto $H$. Since $x\in \Ran(\Phi)$, we can write $x$ as \[x = (x_\alpha)_\alpha,\quad x_\alpha\in X_\alpha,\] and $p_H = \operatorname{s}(x) = (\operatorname{s}(x_\alpha))_\alpha$. For each $\alpha$ let $p_\alpha = \operatorname{s}(x_\alpha)\in B(\H_\alpha)$. Then \[0<p_\alpha < \Id_{\mathcal{H}_\alpha}.\] Indeed, we have the contractive projection $\tilde{P}:=\left[\begin{array}{cc}
    0 & P \\
    0 & 0
\end{array}\right] : S^p(H\overset{2}{\oplus}K)\to S^p(H\overset{2}{\oplus}K)$ with range \[\Ran(\tilde{P}) = p_H\Ran(\Phi)(\Id_\mathcal{H}-p_H) = \bigoplus^p_\alpha p_\alpha X_\alpha (\Id_{\mathcal{H}_\alpha}-p_\alpha).\] Since $\Ran(P)$ is non-degenerate, we have \[\operatorname{s}_\ell(\tilde{P}) = p_H\quad \text{and}\quad \operatorname{s}_r(\tilde{P}) = \Id_{\mathcal{H}}-p_H.\]
Assume that there exists $\beta$ such that $p_\beta = 0$. Then $p_\beta X_\beta (\Id_{\mathcal{H}_\beta}-p_\beta) = \{0\}.$
Define $q = (q_\alpha)_\alpha \in B(\H)$ where $q_\alpha = \left\{\begin{array}{c}
    1-p_\alpha \text{  if }\alpha\neq \beta\\
    0 \text{  otherwise.}
\end{array}\right.$ Then $q$ is an orthogonal projection such that $q<\Id_\mathcal{H}-p_H$, and for every $x\in\Ran(\tilde{P})$, $xq=x$. Hence $\Id_\mathcal{H} - p_H = \operatorname{s}_r(\tilde{P}) \leq q$, which is a contradiction.
Similarly, if there exists $\beta$ such that $p_\beta = 1$, we can define $q\in B(\H)$ an orthogonal projection such that $q<p_H$ and $p_H\leq q$, which is a contradiction. 

Now, let us show that for every $\alpha$, we have \begin{equation}\label{Phi_commute_avec_Lp}
    \Phi_\alpha \circ L_{p_\alpha} = L_{p_\alpha}\circ \Phi_\alpha.
\end{equation} Then, the result will follow, by Lemma \ref{corner_preserv}. For a fixed $\alpha$, recall that $0<p_\alpha := \operatorname{s}(x_\alpha) < \Id_{\mathcal{H}_\alpha}$. Let $z_\alpha\in S^p(\H_\alpha)$, and define $z := (z_\beta)_\beta$, where $z_\beta = 0$ if $\beta\neq \alpha$. On one hand, \begin{equation*}\begin{split}
     (\Phi\circ L_{p_H})(z) = \Phi\left((p_\beta z_\beta)_\beta\right) = (0,\ldots,0,\Phi_\alpha(p_\alpha z_\alpha),0,\ldots).
\end{split}
\end{equation*} On the other hand, by Lemma \ref{corner_preserv}, \begin{equation*}\begin{split}(\Phi\circ L_{p_H})(z) = (L_{p_H}\circ\Phi)(z) = p_H\left(\Phi(z)\right) &= p_H\left(0,\ldots,0,\Phi_\alpha(z_\alpha),0,\ldots\right)\\ &= (0,\ldots,0,p_\alpha\Phi_\alpha( z_\alpha),0,\ldots).\end{split}
\end{equation*}
We deduce (\ref{Phi_commute_avec_Lp}).

\end{proof}

\begin{lemma}\label{Lemme1_spin}Let $H$ be a Hilbert space, let $n\geq 1$ an integer. Let $(w_1,\ldots, w_n)$ be a spin system in $B(H)$. Define $J_n :=\operatorname{span}\{1,w_1,\ldots,w_n\}$.
Every unitary $u\in J_n$ may be written as \[u = C(\gamma_0 1 + i \sum_{k=1}^n \gamma_k w_k),\] where $C\in\mathbb C^*$, and $\gamma_0,\ldots,\gamma_n\in\Rdb$.

If $u$ is, in addition, selfadjoint, then \[u=1,\quad \text{or}\quad u = \sum_{k=1}^n \gamma_k w_k.\]
\end{lemma}
\begin{proof}
Let $u\in J_n$ be a unitary. There exist $(\alpha_0,\ldots,\alpha_n)\in\mathbb{C}^{n+1}$ such that \[u = \alpha_0 1 + \sum_{k=1}^n \alpha_k w_k.\] Then using the properties of a spin system (see Definition \ref{def_spin}),  \begin{align*}
    1 = u^*u &= \left(\overline{\alpha_0} 1 + \sum_{k=1}^n \overline{\alpha_k} w_k\right)\left(\alpha_0 1 + \sum_{k=1}^n \alpha_k w_k\right)\\ &= \left(\sum_{k=0}^n |\alpha_k|^2\right) 1+ \sum_{k=1}^n (\overline{\alpha_0}\alpha_k + \alpha_0\overline{\alpha_k})w_k + \sum_{1\leq k\neq j\leq n}\overline{\alpha_k}\alpha_j w_k w_j\\ &= \left(\sum_{k=0}^n |\alpha_k|^2\right) 1+ 2\sum_{k=1}^n \Re(\overline{\alpha_0}\alpha_k) w_k + \sum_{1\leq k< j\leq n}\underbrace{(\overline{\alpha_k}\alpha_j - \alpha_k\overline{\alpha_j})}_{= 2\Im(\overline{\alpha_k}\alpha_j)} w_k w_j.
\end{align*}

This implies \begin{equation}\label{condi_unit_spin}
    \left\{\begin{array}{cccc}
    \Re(\overline{\alpha_0}\alpha_k) & =&0,&\quad k\geq 1\\
    \Im(\overline{\alpha_k}\alpha_j)& = &0,&\quad 1\leq k<j\leq n.
\end{array}\right.
\end{equation}

\begin{itemize}
    \item If $\alpha_0\neq 0$, let $\lambda := \dfrac{\overline{\alpha_0}}{|\alpha_0|}$. Then $\lambda u = |\alpha_0|1 + \sum_{k=1}^n \frac{\overline{\alpha_0}\alpha_k}{|\alpha_0|}w_k$. By (\ref{condi_unit_spin}), for every $k\geq 1$, there exists $\gamma_k \in \Rdb$ such that $\frac{\overline{\alpha_0}\alpha_k}{|\alpha_0|} = i\gamma_k$, and \[u = \overline{\lambda}\left(\gamma_0 1 + i \sum_{k=1} \gamma_k w_k\right),\] with $\gamma_0 = |\alpha_0|\in\Rdb.$
    
    If, in addition, $u$ is selfadjoint, since the $w_k$'s are selfadjoint, we deduce that for every $k\geq 1$, $\left\{\begin{array}{ccc}
        \lambda\gamma_0 &=& \overline{\lambda}\gamma_0  \\
        -\lambda\gamma_k &= & \overline{\lambda}\gamma_k 
    \end{array}\right..$ Since $\gamma_0\neq 0$, we deduce $\lambda\in\Rdb^*$ and the second line implies that $\gamma_k=0$ for every $k\geq 1$, so $u=1$.
    
    \item If $\alpha_0=0$, reordering the $w_j$'s, we may assume that $\alpha_1\neq 0.$ Let $\lambda := \dfrac{\overline{\alpha_1}}{|\alpha_1|}$. Then $\lambda u = |\alpha_1|w_1 + \sum_{k=2}^n \frac{\overline{\alpha_1}\alpha_k}{|\alpha_1|}w_k.$ By (\ref{condi_unit_spin}), for every $k\geq 2$, $\gamma_k:=\frac{\overline{\alpha_1}\alpha_k}{|\alpha_1|}$ is real and
     \[ u = \overline{\lambda}\left(\sum_{k=1}^n \gamma_k w_k\right),\] with $\gamma_1 = |\alpha_1|\in\Rdb.$
    
    If, in addition $u$ is selfadjoint, we have $\alpha_1 = \overline{\alpha_1}$, so $\lambda\in\Rdb$ and we have the desired result.
    \end{itemize}

\end{proof}

\begin{lemma}\label{Lemme2_spin}
Let $H$ be a Hilbert space, let $n\geq 1$ an integer. Let $(w_1,\ldots, w_n)$ be a spin system in $B(H)$. Define $J_n :=\operatorname{span}\{1,w_1,\ldots,w_n\}$. Let $w\in J_n$ be a selfadjoint unitary.
For every $z\in J_n$ such that $wz\in J_n$, we have \[zw = wz.\]
\end{lemma}
\begin{proof}
By Lemma \ref{Lemme1_spin}, we can write \[w = 1,\quad \text{or}\quad w=\sum_{k=1}^n \gamma_k w_k,\] with $\gamma_k\in\Rdb.$ 
Assume that $w\neq 1$, let $z\in J_n$ such that $zw\in J_n$. There exist $(\alpha_k)_{0\leq k\leq n}\in\Cdb^{n+1}$ such that \[z = \alpha_0 1 + \sum_{k=1}^n \alpha_k w_k.\]
Then \[ zw = \alpha_0 w + \left(\sum_{k=1}^n \alpha_k w_k\right)\left(\sum_{k=1}^n \gamma_k w_k\right) = \alpha_0w + \left(\sum_{k=1}^n \alpha_k\gamma_k\right) 1 + \sum_{1\leq k<j\leq n} (\alpha_k\gamma_j - \alpha_j\gamma_k)w_kw_j.\] We deduce that for every $1\leq j\neq k\leq n$, $\alpha_k\gamma_j - \alpha_j\gamma_k = 0$. Then,
 \begin{align*}
    wz = \alpha_0 w + \left(\sum_{k=1}^n \gamma_k w_k\right)\left(\sum_{k=1}^n \alpha_k w_k\right) &= \alpha_0w + \left(\sum_{k=1}^n \gamma_k\alpha_k\right) 1 + \sum_{1\leq k<j\leq n} \underbrace{(\alpha_j\gamma_k - \alpha_k\gamma_j)}_{=0}w_kw_j\\ &= \alpha_0w + \left(\sum_{k=1}^n \gamma_k\alpha_k\right) 1\\ &= zw.
\end{align*}
\end{proof}

\begin{proposition}\label{descriptiondephi_cornerpreserving}Let $\H$ be a Hilbert space, let $1<p<\infty$, and let $\Phi : S^p(\H)\to S^p(\H)$ be a non-zero contractive and positive projection. We assume that $\Phi$ is $M_2$-splitting, \[\Phi = \left[\begin{array}{cc}
    P_1 & P \\
    P^\circ & P_2
\end{array}\right] : S^p(H\overset{2}{\oplus}K)\to S^p(H\overset{2}{\oplus}K),\] with $\H = H\overset{2}{\oplus}K$ and non-zero contractive projections $P_1 : S^p(H)\to S^p(H)$, $P_2:S^p(K)\to S^p(K)$, $P : S^p(K,H)\to S^p(K,H)$. We assume, in addition, that $\Ran(\Phi)$ is non-degenerate and indecomposable.

Then, there exist a Hilbert space $K_1$, a countable index set $I$, a positive injective operator $a\in S^p(K_1)$ and a unitary $U : \ell^2_I(K_1)\to \H$ such that \[\Ran(\Phi)= U(S^p_I\otimes a)U^*.\]
\end{proposition}
\begin{proof}
First of all, note that if a subspace $Y$ is positively equivalent to $\Ran(\Phi)$, then $Y$ is indecomposable and positively 1-complemented with a $M_2$-splitting contractive projection. Indeed, let $\mathcal{K}$ be a Hilbert space, let $U : \mathcal{K}\to \H$ be a unitary, and let $Y$ be a subspace of $S^p(\mathcal{K})$ such that \(Y = U^*\Ran(\Phi)U\). According to (3) of Remark \ref{passage_P_a_Q}, $Y$ positively 1-complemented, and it is straightforward to see that $Y$ is indecomposable and non-degenerate. Moreover, we can write \(\mathcal{K} = U^*(H)\overset{2}{\oplus}U^*(K),\) and the map $Q : x\in S^p(\mathcal{K})\mapsto U^*\Phi(UxU^*)U$ is the contractive and positive projection onto $Y$. Let $U_1 : U^*(H)\to H$ and $U_2 : U^*(K)\to K$ be the restrictions of $U$, then it follows that \[Q = \left[\begin{array}{cc}
    U_1^*P_1(U_1\cdot U_1^*)U_1 & U_1^*P(U_1\cdot U_2^*)U_2 \\
    U_2^*P^\circ(U_2\cdot U_1^*)U_1 & U_2^*P_2(U_2\cdot U_2^*)U_2
\end{array}\right] : S^p(U^*(H)\overset{2}{\oplus}U^*(K))\to S^p(U^*(H)\overset{2}{\oplus}U^*(K)), \] hence $Q$ is $M_2$-splitting.

We will describe $\Ran(\Phi)\subset S^p(\H)$. This subspace is positively equivalent to one of the five types of subspaces described in \cite{B} and reviewed in Theorem \ref{main_B}. According to the previous paragraph, we may examine each type, up to the isometry $U$ involved in the equivalence relation $\simp$. We will show that all types except one cannot be the range of a $M_2$-splitting contractive and positive projection.
Note that according to Remark \ref{ran(P)nondegenimpliesran(Phi)nondegen}, since $\Ran(\Phi)$ is non-degenerate, the subspaces $\Ran(P_1)\subset S^p(H)$ and $\Ran(P_2)\subset S^p(K)$ are non-degenerate.

We have $\Phi = \left[\begin{array}{cc}
    P_1 & P \\
    P^\circ & P_2
\end{array}\right] : S^p(H\overset{2}{\oplus}K)\to S^p(H\overset{2}{\oplus}K),$ with non-zero positive contractive projections $P_1 : S^p(H)\to S^p(H)$, $P_2:S^p(K)\to S^p(K)$, and a non-zero contractive projection $P : S^p(K,H)\to S^p(K,H)$. Consider $h\in\Ran(P_1)$ a positive element with $\operatorname{s}(h) = \Id_H$. This element exists according to Proposition \ref{existence_h}. Now, define \begin{equation}\label{def_x1}
    x_1 := \left[\begin{array}{cc}
    h & 0 \\
     0& 0
\end{array}\right]\in S^p(H\overset{2}{\oplus}K).
\end{equation} This element $x_1$ is positive, belongs to the range of $\Phi$, and $\operatorname{s}(x_1) = p_H$, where $p_H\in B(H\overset{2}{\oplus}K)$ is the orthogonal projection onto $H$. 
For each type, we examine $x_1$ to obtain some information about the projection $p_H$. Then, we consider an arbitrary element $z\in\Ran(P)$, and define \begin{equation}\label{def_x}
    x:= \left[\begin{array}{cc}
    0 & z \\
    0 & 0
\end{array}\right]\in\Ran(\Phi).
\end{equation} We will show that if $\Ran(\Phi)$ is not positively equivalent to a space $S^p_I\otimes a$, then 
$x$ and $z$ are zero. This gives $\Ran(P)=\{0\}$, which contradicts the definition of $\Phi$.

\begin{enumerate}
    \item[\textbf{Case 1.}] Let \( I \) be a countable index set.  Let $a\in S^p(H_1)$ be a positive injective operator on a Hilbert space $H_1$, let $O\in \S_I$ be a symmetric unitary operator. We assume that $$\Ran(\Phi) = O\S^p_I\otimes a \subset S^p(\ell^2_I(H_1)).$$
Since $x_1\in\Ran(\Phi)$, there exists $s_1\in\S^p_I$ such that $x_1 = Os_1\otimes a$. Then, since $a$ is positive and injective, we have \[p_H=\operatorname{s}(x_1) = \operatorname{s}(Os_1)\otimes \Id_{H_1}.\]
Now, note that if $y\in S^p_I$, we have \begin{equation}\label{transpo_supp}
    \operatorname{s}_\ell(y^\top) = (\operatorname{s}_r(y))^\top.
\end{equation} Indeed, if $y = u|y|$ is the polar decomposition of $y$, we have $\operatorname{s}_\ell(y)=uu^*$ and $\operatorname{s}_r(y) = u^*u$. Moreover, $y^\top= |y|^\top u^\top$, hence $(y^\top)^* y^\top = (u^\top)^* (|y|^\top)^2 u^\top$, from which we deduce \[|y^\top| = (u^\top)^* |y|^\top u^\top.\] We also have that $u^\top$ is a partial isometry and that \begin{equation}\label{dec_pol_y}
    y^\top = u^\top|y^\top|. 
\end{equation} Moreover, $\ker(u^\top) = \ker(y^\top)$.
We deduce that the equality (\ref{dec_pol_y}) is the polar decomposition of $y$ and \[\operatorname{s}_\ell(y^\top) = u^\top (u^\top)^* = (u^*u)^\top = (\operatorname{s}_r(y))^\top.\]
Set $p_1 := \operatorname{s}(Os_1)$. Then, we have \[p_1 = \operatorname{s}_\ell(Os_1) = O\operatorname{s}_\ell(s_1)O^*,\quad \text{and  } p_1 = \operatorname{s}_r(Os_1) = \operatorname{s}_r(s_1).\]
We deduce, using the equality $s_1^\top =s_1$, \begin{equation}\label{p_1dansOsym}
    p_1^\top = (\operatorname{s}_r(s_1))^\top = \operatorname{s}_\ell(s_1) = O^* p_1O.
\end{equation}
Consider the element $x\in\Ran(\Phi)$ defined by (\ref{def_x}). There exists $s\in\S^p_I$ such that $x = Os\otimes a$, and we have $x = p_Hx(\Id_\mathcal{H}-p_H) = (p_1\otimes \Id_{H_1})(Os\otimes a)((\Id_{\ell^2_I}-p_1)\otimes \Id_{H_1}) = p_1(Os)(\Id_{\ell^2_I}-p_1)\otimes a$. Then, we have \begin{equation}\label{Os1}
    Os = p_1(Os)(\Id_{\ell^2_I}-p_1).
\end{equation} This implies, taking the transpose of this equality, that \[sO = (\Id_{\ell^2_I}-p_1^\top)sO p_1^\top \overset{(\ref{p_1dansOsym})}{=} (\Id_{\ell^2_I} - O^*p_1O)sOO^*p_1O.\] Hence \(sO = O^*(\Id_{\ell^2_I} - p_1)Osp_1O\), and we deduce \begin{equation}\label{Os2}
    Os = (\Id_{\ell^2_I}-p_1)(Os)p_1.
\end{equation} The identities (\ref{Os1}) and (\ref{Os2}) imply $Os = 0$, so $x=0$. Then, the projection $P$ is zero.

\item[\textbf{Case 2.}] Let $I$ be a countable index set.
Let $a\in S^p(H_1)$ be a positive injective operator on a Hilbert space $H_1$, and let $O\in\A_I$ be an anti-symmetric unitary operator. We assume that \[\Ran(\Phi) = O\A^p_I\otimes a \subset S^p(\ell^2_I(H_1)).\] Using the same approach as before, we obtain that the orthogonal projection onto $H$ is \[p_H = p_1\otimes \Id_{H_1}\] with $p_1\in B(\ell^2_I)$ an orthogonal projection such that \begin{equation}\label{equ_on_p1}
    p_1^\top = O^*p_1O.
\end{equation}
Then, since the element $x= \left[\begin{array}{cc}
    0 & z \\
    0 & 0
\end{array}\right]$ is in $\Ran(\Phi)$, 
there exists $y\in\A^p_I$ such that $x = Oy\otimes a$. Since $p_Hx(\Id_\mathcal{H}-p_H)=x$, we have \[Oy = p_1(Oy)(\Id_{\ell^2_I}-p_1).\] Taking the transpose of this equality and using (\ref{equ_on_p1}), we obtain \[yO = O^*(\Id_{\ell^2_I}-p_1)O(yO)O^*p_1O,\] hence \[O^*(Oy)O = O^*(\Id_{\ell^2_I}-p_1)(Oy)p_1O\] so \[Oy = (\Id_{\ell^2_I}-p_1)(Oy)p_1.\] We deduce that $Oy = 0$, so $x=0$, and $P=0$.

\item[\textbf{Case 3.}]Let $I$ be a countable index set, let $a_1\in S^p(H_1)$, $a_2 \in S^p(H_2)$ be injective and positive operators on Hilbert spaces $H_1$ and $H_2$ respectively. We assume that \[\Ran(\Phi) = \{(w\otimes a_1, w^\top\otimes a_2),\quad w\in S^p_I\}.\]
First, note that if $a_1$ or $a_2$ is zero, the conclusion of Proposition \ref{descriptiondephi_cornerpreserving} follows immediately, we assume henceforth that $a_1$ and $a_2$ are non zero. There exists a positive element $w_1\in S^p_I$ such that the element $x_1$ defined by (\ref{def_x1}) may be written as \(x_1 = (w_1\otimes a_1, w_1^\top\otimes a_2).\) Then the orthogonal projection onto $H$ may be written as \[p_H = \operatorname{s}(x_1) = (\operatorname{s}(w_1)\otimes \Id_{H_1}, \operatorname{s}(w_1^\top)\otimes \Id_{H_2}) \overset{(\ref{transpo_supp})}{=} (p_1\otimes \Id_{H_1}, p_1^\top\otimes \Id_{H_2}),\] where $p_1 = \operatorname{s}(w_1)$. Now, consider $w\in S^p_I$ such that the element $x$ defined by (\ref{def_x}) may be written as \(x = (w\otimes a_1, w^\top\otimes a_2)\). Since $p_Hx(\Id_\mathcal{H}-p_H)=x$, we have \[(p_1w(\Id_{\ell^2_I}-p_1)\otimes a_1, p_1^\top w^\top(\Id_{\ell^2_I}-p_1^\top)\otimes a_2) = (w\otimes a_1, w^\top\otimes a_2).\]
We deduce that \[\left\{\begin{array}{ccc}
    w & = &p_1w(\Id_{\ell^2_I}-p_1) \\
    w^\top& = &p_1^\top w^\top(\Id_{\ell^2_I}- p_1^\top) 
\end{array}\right. .\] So \[\left\{\begin{array}{c}
    w = p_1w(\Id_{\ell^2_I}-p_1) \\
    w = (\Id_{\ell^2_I}-p_1)wp_1
\end{array}\right. ,\] hence \(w=0\) and \(x=0.\) This implies that $P=0$.

\item[\textbf{Case 4.}] Assume that \[\Ran(\Phi) = v\mathcal{E}^p_{2N}\otimes a,\] for $N\geq 2$ an integer, $a\in S^p(H_1)$ a positive injective operator on a Hilbert space $H_1$, and $v\in\mathcal{E}_{2N}$ a unitary.
There exists $e_1\in \mathcal{E}^p_{2N}$ such that $x_1 = ve_1\otimes a$, and we have \[p_H = \operatorname{s}(x_1) = \operatorname{s}(ve_1)\otimes \Id_{H_1}.\] Let $p_1:=\operatorname{s}(ve_1)$. By definition, the space $\mathcal{E}_{2N}$ is $*$-invariant, we deduce that $v^*\in \mathcal{E}_{2N}$, hence the element $1\otimes a\in\Ran(\Phi)$. Since $\Phi = \left[\begin{array}{cc}
    P_1 & P \\
    P^\circ & P_2
\end{array}\right]$,
We can write $1\otimes a := \left[\begin{array}{cc}
    \alpha_1 & \alpha_2 \\
    \alpha_3 & \alpha_4
\end{array}\right]\in S^p(H\overset{2}{\oplus}K)$, with $\alpha_1\in \Ran(P_1)$, $\alpha_2\in \Ran(P)$, $\alpha_3\in \Ran(P^\circ)$ and $\alpha_4\in \Ran(P_2)$. Then \[p_1\otimes a = p_H(1\otimes a)p_H = \left[\begin{array}{cc}
    \alpha_1 &  0\\
    0 & 0
\end{array}\right]\in \Ran(\Phi) = v\mathcal{E}_{2N}\otimes a.\] We deduce that $$p_1\in v\mathcal{E}_{2N}.$$
Since $\Phi$ is $M_2$-splitting, with $\Phi = \left[\begin{array}{cc}
   P_1  & P \\
   P^\circ  & P_2
\end{array}\right]$, we have \[L_{p_H}\circ \Phi = \Phi\circ L_{p_H},\] see Lemma \ref{corner_preserv}. We deduce that \[L_{\Id_{\mathcal{H}}-2p_H}\circ \Phi = \Phi\circ L_{\Id_{\mathcal{H}}-2p_H}.\] Since $\Id_\mathcal{H}-2p_H$ is a unitary, this implies \[(\Id_\mathcal{H} - 2p_H)\Ran(\Phi) = \Ran(\Phi),\] hence \[(1-2p_1)v\mathcal{E}_{2N}\otimes a = v\E_{2N}\otimes a.\]
We deduce that \[(1-2p_1)v\E_{2N} = v\E_{2N}.\]
Note that $\E_{2N}$ and $v\E_{2N}$ are unital $JC^*$-triples, so their are $*$-invariant subspaces (see \cite[Lemma 3.3]{B}). Since $1\in\E_{2N}$, the elements $v$ and $v^*$ are in $v\E_{2N}$. We deduce that $(1-2p_1)v^*\in v\E_{2N}$, hence there exists $e\in\E_{2N}$ such that $(1-2p_1)v^* = ve$. This implies \[1-2p_1 = vev.\] Since $\E_{2N}$ is a $*$-invariant $JC^*$-triple, and $v\in \E_{2N}$ is a unitary, we have $v\E_{2N}v=\E_{2N}$. We deduce that the element $w:= 1-2p_1$ is a selfadjoint unitary in $\E_{2N}$. Since $p_1\in v\E_{2N}$, we also have $w\in v\E_{2N}$, hence $wv\in v\E_{2N}v = \E_{2N}$. According to Lemma \ref{Lemme2_spin}, this implies \begin{equation}\label{wv=vw}
    wv = vw.
\end{equation} Now, consider $x\in\Ran(\Phi)$ defined by (\ref{def_x}), there exists $y\in\E_{2N}$ such that $x = vy\otimes a$. Moreover, $p_Hx(1-p_H) = x$, hence \begin{equation}\label{vt corner}
    p_1(vy)(1-p_1) = vy.
\end{equation} Since $1-2p_1 = w$, $p_1 = \dfrac{1-w}{2}$ and $1-p_1 = \dfrac{1+w}{2}$, so \[\frac{1}{4}((1-w)vy(1+w)) = vy \in v\E_{2N},\] hence \[vy - wvy + vy w -wvyw \in v\E_{2N}.\] But $w, vy\in v\E_{2N}$ which is a $JC^*$-triple, so $vy, w(vy)w$ are in $v\E_{2N}$, hence $wvy - vyw \in v\E_{2N}$. Moreover, $wvy+vyw\in v\E_{2N}$, indeed, $1$, $w$ and $(vy)^*$ are in $v\E_{2N}$, hence $w(vy)1 + 1(vy)w\in v\E_{2N}$.
We deduce that $wvy = \frac{1}{2}\left((wvy+vyw)-(vyw-wvy)\right)\in v\E_{2N}$. Hence $wvy \underbrace{=}_{(\ref{wv=vw})}vwy \in v\E_{2N}$. We deduce that $wy\in\E_{2N}$, and by Lemma \ref{Lemme2_spin}, we have $wy=yw$. Combined with (\ref{wv=vw}), $w$ commutes with $vy$, so $p_1 = \frac{1-w}{2}$ commute with $vy$, and using (\ref{vt corner}), we conclude that $vy=0$, and hence $x=0$. So $P=0$.

\item[\textbf{Case 5.}] Let $a_1\in S^p(H_1)$ and let $a_2\in S^p(H_2)$ be positive and injective operators on Hilbert spaces $H_1, H_2$. Let $N\geq 2$ an integer, let $v\in\F_N$ be a unitary. Assume that \[\Ran(\Phi) = \left\{(vf\otimes a_1, \sigma(v)\sigma(f)\otimes a_2)~|~f\in\F_N\right\},\] with $\sigma : \F_N\to \F_N$ defined by (\ref{defsigma}). We may assume that $a_1\neq 0$. Indeed, if $a_1=0$, then $a_2\neq 0$ and the reasoning is similar. There exists $f_1\in\F_N$ such that the element $x_1$ defined by (\ref{def_x1}) could be written as $x_1 = (vf_1\otimes a_1, \sigma(v)\sigma(f_1)\otimes a_2)$, and we have \[p_H = \operatorname{s}(x_1) = (\operatorname{s}(vf_1)\otimes \Id_{H_1}, s\left(\sigma(v)\sigma(f_1)\right)\otimes \Id_{H_2}).\] Let $p_1:= \operatorname{s}(vf_1)$ and let $p_2:=\operatorname{s}(\sigma(v)\sigma(f_1))$.
Since $v$ and $v^*$ are in $\F_N$, we have \[(1\otimes a_1, 1\otimes a_2)\in \Ran(\Phi).\] As in Case 5, we have \[(p_1\otimes a_1, p_2\otimes a_2) = p_H(1\otimes a_1, 1\otimes a_2)p_H\in\Ran(\Phi).\] We deduce that \[p_1 = vu_1,\quad p_2 = \sigma(v)\sigma(u_1),\] with $u_1\in \F_N$.
Now, as in Case 5, $L_{p_H}\circ \Phi = \Phi\circ L_{p_H}$ and we deduce that \[(\Id_\mathcal{H} - 2p_H)\Ran(\Phi) = \Ran(\Phi).\] 
Since, similarly to the case 5, $v$ and $v^*$ are in $v\F_N$, we deduce that there exists $f\in\F_N$ such that \((1-2p_1)v^*\otimes a_1  =  vf\otimes a_1\) . Since $\F_N$ is a $JC^*$-triple, we obtain \(
    1-2p_1 = vfv \in \F_N.\) We have $w:= 1-2p_1\in \F_N\cap v\F_N$ a selfadjoint unitary. Then, we deduce that $wv\in v\F_Nv = \F_N$. By Lemma \ref{Lemme2_spin}, we obtain that $w$ and $v$ commute.
Now, consider $x\in\Ran(\Phi)$ defined by (\ref{def_x}). There exists $y\in\F_N$ such that $x = (vy\otimes a_1, \sigma(v)\sigma(y)\otimes a_2)$. Since $p_Hx(\Id_\mathcal{H}-p_H) = x$, we have \[p_1(vy)(1-p_1) = vy\quad \text{and}\quad p_2\sigma(v)\sigma(y)(1-p_2) = \sigma(v)\sigma(y).\]
By similar calculations as in Case 5, we obtain that $w$ commute with $v$ and $y$, so $p_1$ commutes with $vy$. Finally, we have $vy=0$, hence $y=0$ and $x=0$, which means that $P=0$.

\end{enumerate}
Finally, we obtain that the only case where $\Phi$ could be a non-zero $M_2$-splitting positive and contractive projection with an indecomposable range is the case where $\Ran(\Phi)$ is positively equivalent to the subspace $S^p_I\otimes a$, for $I$ a countable index set, and $a\in S^p(H_1)$ a positive and injective operator.
\end{proof}

\begin{proposition}\label{finpreuvemain2}
    Let $H,K$ be Hilbert spaces, let $1<p<\infty$, let $X$ be a non-trivial 1-complemented subspace of $S^p(K,H)$, with a contractively 1-pseudo decomposable projection. Then the contractive projection onto $X$ is contractively decomposable, and there exist a set $A$, two families of indices $(I_\alpha)_{\alpha\in A}$ and $(J_\alpha)_{\alpha \in A}$, a family of Hilbert spaces $(H_\alpha)_{\alpha\in A}$, some positive and injective operators $a_\alpha \in S^p(K_\alpha)$ such that\[X\sim \bigoplus_{\alpha\in A}^p S^p_{I_\alpha,J_\alpha}\otimes a_\alpha.\]
\end{proposition}
\begin{proof}
Note that according to Remark \ref{nondegen_subs}, we may assume that $X$ is non-degenerate. Denote by $P:S^p(K,H)\to S^p(K,H)$ the contractively 1-pseudo decomposable projection onto $X$.
According to Lemma \ref{phi_proj_Sp}, there exist positive contractive projections $P_1 : S^p(H)\to S^p(H)$, $P_2:S^p(K)\to S^p(K)$ with non-degenerate ranges such that the linear map $$\Phi := \left[\begin{array}{cc}
    P_1 & P \\
    P^\circ & P_2
\end{array}\right] : S^p(H\overset{2}{\oplus}K)\to S^p(H\overset{2}{\oplus}K)$$ is a contractive, positive and $M_2$-splitting projection.
Then, combining Lemma \ref{phi_diago_de_cornerpreserving} and Proposition \ref{descriptiondephi_cornerpreserving}, we describe the range of $\Phi$ as: \[\Ran(\Phi) = \bigoplus^p_\alpha U_\alpha\left(S^p_{I_\alpha}\otimes a_\alpha\right)U_\alpha^*,\] with $(I_\alpha)_\alpha$ a family of index sets, $a_\alpha \in S^p(H_\alpha)$ some positive and injective operators on Hilbert spaces $H_\alpha$, and $U_\alpha : \ell^2_{I_\alpha}(H_\alpha)\to H\overset{2}{\oplus}K$ an isometry, for every $\alpha$. Note that according to Lemma \ref{dec_space} and Proposition \ref{projcp}, $\Phi$ is completely positive, hence $P$ is contractively decomposable.

The subspaces $\Ran(P)\subset S^p(K,H)$ and $\left[\begin{array}{cc}
    0 & \Ran(P) \\
    0 & 0
\end{array}\right]\subset S^p(H\overset{2}{\oplus}K)$ are equivalent, and \[\left[\begin{array}{cc}
    0 & \Ran(P) \\
    0 & 0
\end{array}\right] = p_H\Ran(\Phi)(I_\mathcal{H} - p_H),\] where $\H := H\overset{2}{\oplus}K$ and $p_H\in B(\H)$ is the orthogonal projection onto $H$.
Now, consider $h\in\Ran(P_1)$ a positive element such that $\operatorname{s}(h) = \Id_H$. This element exists according to Proposition \ref{existence_h}. Let $x_1 := \left[\begin{array}{cc}
    h & 0 \\
    0 & 0
\end{array}\right]$ with $x_1\geq 0$ and $\operatorname{s}(x_1)=p_H$. Since $h$ is in $\Ran(P_1)$, we deduce that $x_1$ is in $\Ran(\Phi)$. There exists $x_\alpha \in S^p_{I_\alpha}$ for each $\alpha$ such that $x_1 = \left(U_\alpha(x_\alpha\otimes a_\alpha)U_\alpha^*\right)_\alpha$. We deduce that \[p_H = \operatorname{s}(x_1) = \left(U_\alpha(\operatorname{s}(x_\alpha)\otimes \Id_{H_\alpha})U_\alpha^*\right)_\alpha.\] We obtain \begin{equation*}\begin{split}
    p_H\Ran(\Phi)(\Id_\mathcal{H}-p_H) &= \bigoplus^p_\alpha U_\alpha\left(\operatorname{s}(x_\alpha)S^p_{I_\alpha}(\Id_{\ell^2_{I_\alpha}} - \operatorname{s}(x_\alpha)\otimes a_\alpha\right)U_\alpha^*\\ &= U\left(\bigoplus^p_\alpha \operatorname{s}(x_\alpha)S^p_{I_\alpha}(\Id_{\ell^2_{I_\alpha}} - \operatorname{s}(x_\alpha))\otimes a_\alpha\right)U^*,
\end{split}\end{equation*} where $U := \left(U_\alpha\right)_\alpha :  \overset{2}{\underset{\alpha}{\bigoplus}} \ell^2_{I_\alpha}(H_\alpha)\to \H$ is an isometry defined by $U\left((\xi_\alpha)_\alpha\right) := \left(U_\alpha(\xi_\alpha)\right)_\alpha$.
We deduce that $\Ran(P)$ is equivalent to the subspace \[\bigoplus^p_\alpha \operatorname{s}(x_\alpha)S^p_{I_\alpha}(\Id_{\ell^2_{I_\alpha}} - \operatorname{s}(x_\alpha))\otimes a_\alpha.\] If we denote $E_\alpha := \Ran(\operatorname{s}(x_\alpha))$ for any $\alpha$ we have the equivalence of spaces \[\operatorname{s}(x_\alpha)S^p_{I_\alpha}(\Id_{\ell^2_{I_\alpha}} - \operatorname{s}(x_\alpha)) \sim S^p(E_\alpha^\perp, E_\alpha),\] hence \[X =\Ran(P) \sim \bigoplus^p_\alpha S^p(E_\alpha^\perp, E_\alpha)\otimes a_\alpha \sim \bigoplus^p_\alpha S^p_{I^1_\alpha, I^2_\alpha}\otimes a_\alpha,\] where $I^1_\alpha$ and $I^2_\alpha$ are two index sets with the same cardinality as $E_\alpha$ and $E_\alpha^\perp$ respectively, for any $\alpha$.
\end{proof}

\begin{remark}
If we are only looking for contractively decomposable projections in $S^p(K,H)$, we could have a faster way, without passing through contractively 1-pseudo-decomposable projections. First of all, for $1\leq p\neq 2<\infty$, according to Remark \ref{appli_cb}, if $P:S^p(K,H)\to S^p(K,H)$ is a contractively decomposable projection, then it is a completely contractive projection. We use \cite[Theorem 1.1]{MRR} to show that, in this case, \[\Ran(P) \sim \bigoplus_{\alpha\in A}^p S^p_{I_\alpha,J_\alpha}\otimes a_\alpha,\] for $A$ a set, and for any $\alpha\in A$, $I_\alpha$ and $J_\alpha$ are index sets, and $a_\alpha\in S^p(H_\alpha)$ is an operator on a Hilbert space $H_\alpha$. 

In the case $p=2$, if $P:S^2(K,H)\to S^2(K,H)$ is a contractively decomposable projection, then we have a contractive $M_2$-splitting and completely positive projection $\Phi:S^2(H\overset{2}{\oplus}K)\to S^2(H\overset{2}{\oplus}K)$ that can be written as \[\Phi := \left[\begin{array}{cc} P_1 & P\\ P^\circ & P_2\end{array}\right],\] with $P_1:S^2(H)\to S^2(H)$ and $P_2:S^2(K)\to S^2(K)$ some completely positive and contractive projections. Combining Lemma \ref{dec_space} and Proposition \ref{projcp} to $\Ran(\Phi)$, we describe this subspace as \[\Ran(\Phi) = \bigoplus^2_{\alpha\in A} U_\alpha(S^p_{I_\alpha}\otimes a_\alpha)U_\alpha^*,\] with $I_\alpha$ a an index set, $a_\alpha \in S^p(H_\alpha)$ a positive and injective operator on a Hilbert space $H_\alpha$, and $U_\alpha : \ell^2_{I_\alpha}(H_\alpha)\to H\overset{2}{\oplus}K$ an isometry, for every $\alpha\in A$. Then, we recover $\Ran(P)$ similarly to the end of the proof of Proposition \ref{finpreuvemain2}.

\end{remark}

\vspace{2cm}
\textbf{Acknowledgment.} The author thanks Cédric Arhancet, who suggested the question that led to this article, and Éric Ricard, for valuable discussions which helped bring this work to completion.


\end{document}